\documentclass{amsart}
\usepackage[utf8]{inputenc}
\usepackage{amsmath,amsthm, amssymb}
\usepackage{hyperref}
\usepackage{tikz,color}
\usepackage{todonotes}
\usetikzlibrary{calc,math}
\usetikzlibrary{intersections}

\newtheorem{theorem}{Theorem}[section]
\newtheorem{corollary}[theorem]{Corollary}
\newtheorem{proposition}[theorem]{Proposition}
\newtheorem{lemma}[theorem]{Lemma}
\newtheorem{BMlemma}{Lemma}
\newtheorem{conjecture}[theorem]{Conjecture}
\newtheorem{question}[theorem]{Question}

\theoremstyle{definition}
\newtheorem{example}{Example}[section]

\newtheorem{openproblem}[example]{Open Problem}
\newtheorem{strategy}[example]{Strategy}

\DeclareMathOperator{\st}{\mid}

\renewcommand{\star}[2]{star_{#1}(#2)}
\newcommand{\del}[2]{del_{#1}(#2)}
\newcommand{\link}[2]{link_{#1}(#2)}
\newcommand{\aus}[2]{costar_{#1}(#2)}

\newcommand{\SC}[2]{\mathcal{S}(#1,#2)} 
\newcommand{\ausSC}[3]{\aus{\SC{#1}{#2}}{#3}} 
\newcommand{\Dem}[1]{\operatorname{Dem}(#1)} 
\newcommand{\wo}{\omega_\circ}
\newcommand{\rotatedword}[2]{#1_{\stackrel{\circlearrowleft}{#2}}}

\newcommand{\complex}[1]{\langle#1\rangle}

\definecolor{orange}{rgb}{0.898, 0.621, 0.0}
\definecolor{skyblue}{rgb}{0.336, 0.703, 0.910}
\definecolor{bluishgreen}{rgb}{0, 0.617, 0.449}
\definecolor{yellow}{rgb}{0.937, 0.890, 0.258}
\definecolor{blue}{rgb}{0, 0.445, 0.695}
\definecolor{red}{rgb}{0.832, 0.367, 0}
\definecolor{purple}{rgb}{0.797, 0.473, 0.652}

\newcommand{\JSD}[1]{\todo[size=\footnotesize,color=skyblue,inline]{#1  \hfill --- J.}} 
\newcommand{\cesar}[1]{\todo[size=\footnotesize,color=orange!30,inline]{#1  \hfill --- C.}}

\newcommand{\defn}[1]{{\color{green!50!black}\emph{#1}}}

\title{Subword Complexes and Kalai's Conjecture on Reconstruction of Spheres}
\author{Cesar Ceballos}
\address{CC: TU Graz, Institut f\"ur Geometrie, Kopernikusgasse 24, 8010 Graz, Austria.}
\email{cesar.ceballos@tugraz.at}
\author{Joseph Doolittle}
\address{JD: TU Graz, Institut f\"ur Geometrie, Kopernikusgasse 24, 8010 Graz, Austria.}
\email{jdoolittle@tugraz.at}
\thanks{Both authors were supported by the Austrian Science Fund FWF, grant P 33278.}

\begin{document}

\maketitle

\begin{abstract}
    A famous theorem in polytope theory states that the combinatorial type of a simplicial polytope is completely determined by its facet-ridge graph. 
    This celebrated result was proven by Blind and Mani in 1987, via a non-constructive proof using topological tools from homology theory. 
    An elegant constructive proof was given by Kalai shortly after. 
    In their original paper, Blind and Mani asked whether their result can be extended to simplicial spheres, and a positive answer to their question was conjectured by Kalai in 2009. 
    In this paper, we show that Kalai’s conjecture holds in the particular case of Knutson and Miller’s spherical subword complexes. 
    This family of simplicial spheres arises in the context of Coxeter groups, and is conjectured to be polytopal.
    In contrast, not all manifolds are reconstructible. 
    We show two explicit examples, namely the torus and the projective plane.
\end{abstract}

\section{Introduction}
The combinatorial structure of a simple polytope is known to be completely determined by its graph.
In other words, one can determine all the faces of a simple polytope knowing only information about its vertices and the edges connecting them. 
This astonishing and beautiful result was originally raised as a conjecture by Perles during conversations at Oberwolfach meetings on convex bodies in Germany in 1984 and 1986, see~\cite{blind_puzzles_1987}. 
The result was then proved by Blind and Mani in their 1987 seminal paper~\cite{blind_puzzles_1987}, where they give a non-constructive proof using topological tools from homology theory in the dual context of simplicial polytopes. 

\begin{theorem}[Blind and Mani~\cite{blind_puzzles_1987}]
Let \(P,Q\) be \(d\)-dimensional simple polytopes. Every isomorphism between their graphs \(f: G(P) \to G(Q)\) has a unique extension to a face isomorphism \(g: P \to Q\).
\end{theorem}

Shortly after, a simple and elegant constructive proof was given by Kalai in his work ``A simple way to tell a simple polytope from its graph"~\cite{kalai_simple_1988}. 
Kalai's proof uses an exponential time algorithm. The complexity of algorithms for reconstructing a simple polytope from its graph (or equivalently, a simplicial polytope from its facet-ridge graph) is studied in~\cite{joswig_Ksystems_2002} and extended to a polynomial time algorithm for determining the \((d-2)\)-faces in~\cite{friedman_finding_2009}.  
One natural question is whether Blind and Mani's result also holds for simplicial spheres. 
This was posed as a question by Blind and Mani in~\cite[Question~1]{blind_puzzles_1987}. 
A positive answer to this question was conjectured by Kalai in his more recent online blog~\cite{kalai_blog_2009}. 

\begin{conjecture}[Blind and Mani~\cite{blind_puzzles_1987}, Kalai~\cite{kalai_blog_2009}]
\label{conj_kalai}
Let \(P,Q\) be \(d\)-dimensional simplicial spheres. Every isomorphism between their facet-ridge graphs \(f: FR(P) \to FR(Q)\) has a unique extension to a simplicial isomorphism \(g: P \to Q\).
\end{conjecture}

The classification and comparison of simplicial spheres and simplicial polytopes is an interesting and challenging topic on its own. It is well known, by Steinitz's theorem, that every 2-dimensional sphere can be realized as the boundary complex of a 3-dimensional polytope~\cite[Chapter~13]{grunbaum_convexpolytopes}.  
The smallest examples of non-polytopal spheres area the Barnette sphere~\cite{barnette_sphere_1973} and the Br\"uckner--Gr\"unbaum--Sreedharan sphere~\cite{grunbaum_sphere_1967}. 
Both of these are 3-dimensional simplicial spheres with 8 vertices which are not the boundary of a 4-dimensional polytope (but both are reconstructible from their facet-ridge graphs using Kalai's method from~\cite{kalai_simple_1988}).

In~\cite{many_spheresfewvertices_1972}, Mani showed that every simplicial $d$-sphere with $d+4$ vertices can be realized as the boundary of a $(d+1)$-polytope. 
The enumeration of simplicial 3-spheres (some of which are not polytopal) with up to 10 vertices can be found in~\cite{lutz_manifoldsfewvertices}; see also~\cite{firsching_completeenumeration_2020} for an overview of more recent results.
Examples on non-polytopal spheres which are non-constructible, and in particular non-shellable, using knot constrains appear in~\cite{hachimori_ziegler_spheresknots_2000,lutz_nonconstructiblespheres_2004}. 
Interestingly, starting in dimension $d=3$, most $d$-spheres are not polytopal.
For $d>3$, this follows from results of Goodman and Pollack from 1986~\cite{goodman_pollack_asymptoticallyfewerpolytopes_1986,goodman_pollack_upperbounds_1986} and Kalai from 1988~\cite{kalai_manyspheres_1988}, who showed upper and lower bounds for the number of simplicial polytopes and simplicial spheres, respectively. For $d=3$, the result follows from the lower bound for the number of simplicial 3-spheres shown by  Pfeifle and Ziegler in 2004~\cite{pfeifle_ziegler_manythreespheres}.
Given the substantial differences between simplicial spheres and simplicial polytopes, one could guess that Kalai's conjecture on the reconstruction of spheres from their facet ridge-graph might be false.
One possibility to look for counterexamples to Kalai's conjecture is to look at explicit families of non-polytopal spheres. One special class of simplicial spheres which are conjectured to be polytopal and which have received a lot of attention in recent years are Knutson and Miller's spherical subword complexes.

Knutson and Miller introduced the concept of subword complexes first in type~$A$, in connection to their study of Gröbner geometry of Schubert varieties~\cite{knutson_grobner_2005}, and then generalized them in the general context of Coxeter groups in~\cite{knutson_subword_2004}. 
Subword complexes appear in many different contexts and have connections to toric and Bott-Samelson varieties~\cite{escobar_bott_2014,escobar_brick_2016,escobar_toric_2016}, 
cluster algebras~\cite{ceballos_denominator_2015}, 
Hopf algebras~\cite{bergeron_hopf_2017}, 
combinatorics and discrete geometry~\cite{armstrong_sorting_2011,ceballos_subword_2014,ceballos_vtamari_subword_2020,escobar_subword_rootpolytopes_2018,jahn_minkowski_2021,pilaudpocchiola_multitraingulations_2012,pilaudstump_brick_2015,pilaud_barycenter_2015,serrano_maximal_2012,stump_newperspective_2011}.
They also contain diverse families of simplicial complexes of interest, including 
boundaries of polytopes such as 
cyclic polytopes~\cite{ceballos_subword_2014}, 
duals of associahedra and generalized associahedra~\cite{ceballos_subword_2014,chapoton_generalizedssociahedra_2002,fomin_ysystems_2003,pilaudstump_brick_2015},
duals of pointed-pseudotriangulation polytopes~\cite{pilaudpocchiola_multitraingulations_2012,rote_expansive_2003}, 
and 
simplicial multiassociahedra~\cite{jonsson_generalized_2005,pilaudpocchiola_multitraingulations_2012} (the last only conjectured to be polytopal).
A fundamental result in the theory of subword complexes is that they are topological balls or spheres~\cite{knutson_subword_2004}. Knutson and Miller asked whether spherical subword complexes can be realized as boundary complexes of polytopes~\cite{knutson_subword_2004}, and this was conjectured to be the case in~\cite{ceballos_subword_2014}.  
An outstanding and difficult open problem is to find polytopal constructions of spherical subword complexes. 
This has been done in some cases, but remains open in general, see e.g.~\cite{ceballos_fan_2015,labbe_combinatorialfoundations_2020,manneville_fan_2018} and the references therein for more information in this regard. 
We refer to~\cite{ceballos_subword_2014} for a more thorough examination on subword complexes, in particular Section 6.

We began this project hoping to disprove the two conjectures at the same time. 
If there was a pair of non-isomorphic spherical subword complexes with the same facet-ridge graph, then Kalai's conjecture and the polytopality conjecture of spherical subword complexes would both be false. 
Indeed, on one hand, we would have a direct counterexample to Kalai's conjecture.
On the other hand, since Kalai's conjecture does hold for the special case of polytopes, at most one of the two subword complexes could be polytopal, indirectly disproving the polytopality conjecture.
We did not find two such subword complexes. Instead, 
we proved that spherical subword complexes of finite type satisfy Kalai's conjecture.

\begin{theorem}[c.f. Theorem~\ref{thm_SC_reconstruction}]
Spherical subword complexes of finite type are completely determined by their facet-ridge graph.
\end{theorem}

Our proof is not constructive and relies on the topological tools developed by Blind and Mani in~\cite{blind_puzzles_1987} with some appropriate modifications. A key ingredient in our proof is Theorem~\ref{thm_strongvertexdec_finite}, which states that spherical subword complexes of finite type are strongly vertex decomposable, and therefore strongly shellable.  
We suspect that the finite type condition is not necessary, see Conjecture~\ref{conj_infinite_two} and Conjecture~\ref{conj_infinite_three}. 

The paper is organized as follows.
In Section \ref{sec:preliminaries}, we cover the required background material. In Section \ref{sec:reconstruction}, we do the necessary strengthening of Blind and Mani's proof. In Section \ref{sec:sphereicalSubword}, we show the required properties of spherical subword complexes. In Section \ref{sec:SubwordComplexReconstruction}, we combine the previous sections to state and prove our main result. 
In Section \ref{sec:nonreconstructible}, we show that the analogue of Kalai's conjecture does not hold for simplicial manifolds in general. 
More precisely, we present two explicit examples of non-isomorphic triangulations of the torus and of the projective plane with the same facet-ridge graph.  
In Section \ref{sec:furtherDirections}, we discuss further directions of potential research.

\section{Preliminaries}\label{sec:preliminaries}
For this paper, we had to pull together ideas from several areas. In this section we reiterate the required background. We split this in two. The first subsection deals with the fundamental properties of simplicial complexes. The second covers the basics of subword complexes.

\subsection{Simplicial complexes}

The central mathematical object in this paper is the simplicial complex. A \defn{simplicial complex} \(\Delta\) on a set \(V\) is a subset of the power set \(2^V\), such that if \(J \subset I\) and \(I \in \Delta\), then \(J \in \Delta\).

The elements of a simplicial complex are the \defn{faces} of the complex. The \defn{dimension} of a face \(I\) is \(\dim{I}=|I|-1\). The \defn{vertices} of a simplicial complex are the \(0\)-dimensional faces. The \defn{edges} are the \(1\)-dimensional faces. The \defn{facets} are those faces which are maximal within \(\Delta\), meaning that there are no faces which properly contain them. We denote the facets of \(\Delta\) as \(F(\Delta)\). The \defn{ridges} of \(\Delta\) are the faces which contain all but one vertex of a facet.

From the definition, all faces of a simplicial complex can be determined by its facets. We give this idea notation, \(\Delta = \langle F(\Delta) \rangle\), where \(\Delta\) is exactly the complex \defn{generated} by the facets of \(\Delta\). A simplicial complex is said to be \defn{pure} if all of its facets are the same dimension. In such case, this is called the dimension of the complex. 

Simplicial complexes provide a common structure for surprisingly broad collection of other mathematical objects. A simplicial complex can be given geometrical information. The standard geometric representation of a simplicial complex \(\Delta\) is denoted \(|\Delta|\) and is defined by
\[|\Delta| := \bigcup_{I \in \Delta} e_I\]
where \(e_I\) is the convex hull of the standard basis vectors \(e_i\) with \(i \in I\). This notation of \(|\Delta|\) should not be confused with the cardinality notation, which we also use throughout this paper. 
Geometric realizations will be restricted to complexes, whereas cardinality will only be applied to faces.
Simplicial complexes can have many geometric realizations, where the standard basis vectors are replaced by another collection of points in space. We say such a realization is a geometric realization as long as the intersection of the convex hull associated to a pair of faces is the convex hull of the intersection of those faces.

The geometric realization gives a topology to simplicial complexes, inherited from the standard topology on euclidean space. 
We define a simplicial complex to be a \defn{spherical} simplicial complex if its geometric realization is homeomorphic to a sphere. 
We also define a simplicial complex to be a \defn{simplicial manifold} if its geometric realization is homeomorphic to a topological manifold.

There are several operations which can be done to a simplicial complex which result in another simplicial complex. 
We list several of these operations, which we will use later in this paper.

Given a simplicial complex \(\Delta\) and a subset of its ground set \(W \subset V\), we define the \defn{induced subcomplex} \(\Delta|_W\). In set notation,
\(\Delta|_W = \{J \in \Delta \st J \subset W\}\).

The \defn{deletion} of a face \(I\) of a simplicial complex \(\Delta\) is the simplicial complex where all faces containing \(I\) are removed. In set notation,
\(\del{\Delta}{I} = \{J \in \Delta \st I \not \subseteq J \}\). When \(I\) is a single vertex, this is equivalent to removing that vertex from the vertex set, \(\del{\Delta}{I} = \Delta|_{V \setminus I}\)

The \defn{star} of a face \(I\) of a simplicial complex \(\Delta\) is the complex generated by the facets containing \(I\). In set notation, 
\(\star{\Delta}{I} = \{J \in \Delta \st I \cup J \in \Delta\}\).

The \defn{link} of a face \(I\) is the set of faces of the star of \(I\) that do not intersect \(I\) non-trivially. In set notation,
\(\link{\Delta}{I} = \{J \in \Delta \st I \cup J \in \Delta, I \cap J = \emptyset\}\).

Our final operation is less standard than those that appear before. We define the \defn{costar} of a face \(I\) to be the complex generated by the facets not containing \(I\). In set notation,
\(\aus{\Delta}{I}= \{J \in \Delta \st \exists \Gamma \in F(\Delta), J \subseteq \Gamma, I \not\subseteq \Gamma\}\).

From a simplicial complex, we can obtain the facet-ridge graph. There are multiple possible definitions of this graph, useful in different circumstances, for our purposes, the most common definition is sufficient. The \defn{facet-ridge} graph of \(\Delta\) is the graph \(FR(\Delta)\) whose vertices are the facets of \(\Delta\), and whose edges are pairs of facets that contain a ridge of both facets.

\begin{figure}
    \centering
    \begin{tikzpicture}[scale=0.4]
    \node at (2,-1) {$\Delta$};
    \coordinate (1) at (0,0);
    \coordinate (2) at (4,0);
    \coordinate (3) at (2,2);
    \coordinate (4) at (2,4);
    \coordinate (5) at (2,6);
    \filldraw[draw=black, fill=gray!20] (1) -- (2) -- (3) -- cycle;
    \filldraw[draw=black, fill=gray!20] (1) -- (3) -- (4) -- cycle;
    \filldraw[draw=black, fill=gray!20] (1) -- (4) -- (5) -- cycle;
    \filldraw[draw=black, fill=gray!20] (2) -- (3) -- (4) -- cycle;
    \filldraw[draw=black, fill=gray!20] (2) -- (4) -- (5) -- cycle;
    \draw (1) -- (2) -- (5) -- (1) -- (4) -- (2) -- (3) -- (1);
    \draw (3) -- (4) -- (5);
    \node at (1) {$\bullet$};
    \node at (2) {$\bullet$};
    \node at (4) {$\bullet$};
    \node at (5) {$\bullet$};
    \node[circle,fill,inner sep=2.5pt,color=blue,label=below:$I$] at (3) {};
    \end{tikzpicture}
    \quad
    \begin{tikzpicture}[scale=0.4]
    \node at (2,-1) {$\del{\Delta}{I}$};
    \coordinate (1) at (0,0);
    \coordinate (2) at (4,0);
    \coordinate (3) at (2,2);
    \coordinate (4) at (2,4);
    \coordinate (5) at (2,6);
    \filldraw[draw=black, fill=gray!20] (1) -- (4) -- (5) -- cycle;
    \filldraw[draw=black, fill=gray!20] (2) -- (4) -- (5) -- cycle;
    \draw (1) -- (2) -- (5) -- (1) -- (4) -- (2);
    \draw (4) -- (5);
    \node at (1) {$\bullet$};
    \node at (2) {$\bullet$};
    \node at (4) {$\bullet$};
    \node at (5) {$\bullet$};
    \end{tikzpicture}
    \quad
    \begin{tikzpicture}[scale=0.4]
    \node at (2,-1) {$\star{\Delta}{I}$};
    \coordinate (1) at (0,0);
    \coordinate (2) at (4,0);
    \coordinate (3) at (2,2);
    \coordinate (4) at (2,4);
    \coordinate (5) at (2,6);
    \filldraw[draw=black, fill=gray!20] (1) -- (2) -- (3) -- cycle;
    \filldraw[draw=black, fill=gray!20] (1) -- (3) -- (4) -- cycle;
    \filldraw[draw=black, fill=gray!20] (2) -- (3) -- (4) -- cycle;
    \draw (1) -- (2);
    \draw (1) -- (4) -- (2) -- (3) -- (1);
    \draw (3) -- (4);
    \node at (1) {$\bullet$};
    \node at (2) {$\bullet$};
    \node at (4) {$\bullet$};
    \node[circle,fill,inner sep=2.5pt,color=blue,label=below:$I$] at (3) {};
    \end{tikzpicture}
    \quad
    \begin{tikzpicture}[scale=0.4]
    \node at (2,-1) {$\link{\Delta}{I}$};
    \coordinate (1) at (0,0);
    \coordinate (2) at (4,0);
    \coordinate (3) at (2,2);
    \coordinate (4) at (2,4);
    \coordinate (5) at (2,6);
    \draw (1) -- (2);
    \draw (1) -- (4) -- (2);
    \node at (1) {$\bullet$};
    \node at (2) {$\bullet$};
    \node at (4) {$\bullet$};
    \end{tikzpicture}
    \quad
    \begin{tikzpicture}[scale=0.4]
    \node at (2,-1) {$\aus{\Delta}{I}$};
    \coordinate (1) at (0,0);
    \coordinate (2) at (4,0);
    \coordinate (3) at (2,2);
    \coordinate (4) at (2,4);
    \coordinate (5) at (2,6);
    \filldraw[draw=black, fill=gray!20] (1) -- (4) -- (5) -- cycle;
    \filldraw[draw=black, fill=gray!20] (2) -- (4) -- (5) -- cycle;
    \draw (2) -- (5) -- (1) -- (4) -- (2);
    \draw (4) -- (5);
    \node at (1) {$\bullet$};
    \node at (2) {$\bullet$};
    \node at (4) {$\bullet$};
    \node at (5) {$\bullet$};
    \end{tikzpicture}
    \caption{Examples of operations on simplicial complexes.}
    \label{fig_examples_concepts_simplicialcomplexes}
\end{figure}
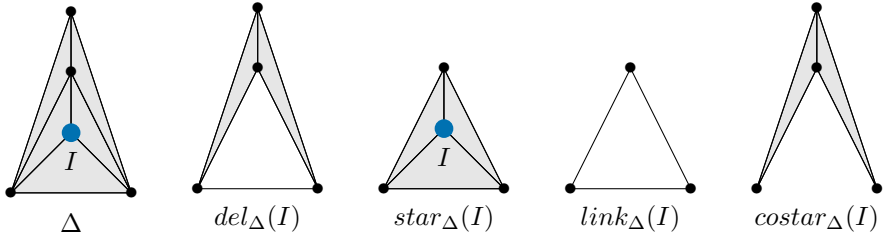

\subsubsection{Shellability and vertex decomposability}

A very important concept in the study of simplicial complexes is that of shellability. A \defn{shellable} simplicial complex \(\Delta\) is either a simplex, or a pure complex with a facet \(F_i\) such that \(\aus{\Delta}{F_i}\) is shellable and \(\langle \{F_i\} \rangle \cap \aus{\Delta}{F_i}\) is pure. 
This definition is one of many equivalent definitions.

For its utility in proofs, we define a variant of shellability used by Blind and Mani in~\cite{blind_puzzles_1987}, called strong shellability. A simplicial complex \(\Delta\) is \defn{strongly shellable} if for any face \(I\), both \(\aus{I}{\Delta}\) and \(\star{I}{\Delta}\) are shellable.

Vertex decomposability is a similar condition to shellability, but quite similar.
A \defn{vertex decomposable} simplicial complex \(\Delta\) is either a simplex, or a pure complex with a vertex \(v_i\) such that \(\del{\Delta}{v_i}\) is vertex decomposable and \(\link{\Delta}{v_i}\) is vertex decomposable.

Similarly to shellability, we can define a stronger condition than vertex decomposable. 
A simplicial complex \(\Delta\) is \defn{strongly vertex decomposable} if for any face \(I\), both \(\aus{I}{\Delta}\) and \(\star{I}{\Delta}\) are vertex decomposable.

These strong definitions imply their base counterparts. In particular, choosing \(I = \emptyset\) gives \(\star{\Delta}{\emptyset} = \Delta\), and \(\aus{I}{\emptyset} = \emptyset\), a simplex. 
Since the simplex is trivially vertex decomposable and shellable, \(\Delta\) having the strong version of these properties implies that \(\Delta\) has the base property as well.

Since vertex decomposability implies shellability~\cite{billera_decomposition_1979}, we have the following straightforward implication. 

\begin{lemma}
Strongly vertex decomposable simplicial complexes are strongly shellable.
\end{lemma}

\subsection{Subword complexes}

Subword complexes are simplicial complexes which naturally arise in the context of Coxeter groups.
Before defining them we recall some basic definitions. We refer to \cite{CoxeterGroups} for a through exposition on Coxeter groups.   

A \defn{Coxeter group} $W$ is a group generated by a collection $S=\{s_1,\dots, s_n\}$ which satisfies the relations $(s_is_j)^{m_{ij}}=1$, where $m_{ii}=1$ and $m_{ij}\geq 2$ for~$i\neq j$.
The pair $(W,S)$ is called a \defn{Coxeter system}.
The \defn{length} of an element $w\in W$ is the smallest~$r$ such that it can be written of the form $w=s_{i_1}s_{i_2}\dots s_{i_r}$. A \defn{reduced expression} of~$w$ is an expression of minimal length.

Coxeter groups generalize the notion of reflection groups, groups generated by a set of reflections of a finite-dimensional Euclidean space. 
A notable example of a Coxeter group is the symmetric group of permutations of $[n+1]$. 
This group is generated by the collection $S=\{s_1,\dots,s_n\}$, where $s_i=(i, i+1)$ is the simple transposition that swaps the numbers $i$ and $i+1$. 
These simple transpositions satisfy the relations $s_i^2=1$, $(s_is_{i+1})^3=1$ (also known as the braid relations) and \mbox{$(s_is_j)^2=1$} whenever $|i-j|>1$. 
Unlike the case of the symmetric group, a Coxeter group associated to some general choice of values $m_{ij}$ is usually infinite. 
In~1935, Coxeter gave a classification of finite Coxeter groups in terms of Dynkin diagrams. 
In this language, the symmetric group corresponds to the type~$A_n$ Coxeter group.

Given a Coxeter system $(W,S)$, we recall certain simplicial complexes introduced by Knutson and Miller in~\cite{knutson_subword_2004}. 
For this, we consider a word $Q=(q_1,\dots,q_r)$ in the generators $S$ and an element $\pi\in W$ of the group. 
For $J\subseteq [r]$, we denote by $Q_J$ the subword of~$Q$ consisting of the letters with positions in $J$. 
The \defn{subword complex} $\SC{Q}{\pi}$ is the simplicial complex whose facets are subsets $I\subseteq [r]$ such that $Q_{[r]\setminus I}$ is a reduced expression for $\pi$. 

\begin{example}\label{ex_subwordcomplex}
Let $W=S_4$ be the symmetric group generated by simple transpositions $s_i=(i, i+1)$ with $i=1,2,3$. 
Let $Q=(q_1,q_2,q_3,q_4,q_5)=(s_1,s_2,s_1,s_2,s_1)$ and $\pi=s_1s_2s_1=s_2s_1s_2$. 
The facets of the subword complex $\SC{Q}{\pi}$ are \[\{1,2\},\{2,3\},\{3,4\},\{4,5\},\{1,5\}.\] 
This complex is depicted on the left of the figure below. 
It is the boundary complex of a pentagon, a one dimensional sphere.

If instead, we consider the word $\widetilde Q=(q_1,q_2,q_3,q_4,q_5,q_6)=(s_1,s_2,s_1,s_2,s_1,s_3)$, no reduced expression of $\pi$ uses the letter $q_6=s_3$, which implies that $6$ belongs to every facet. The facets of $\SC{\widetilde Q}{\pi}$ are  \[\{1,2,6\},\{2,3,6\},\{3,4,6\},\{4,5,6\},\{1,5,6\}.\] The complex is depicted on the right of the figure below, and it is a two dimensional ball.   

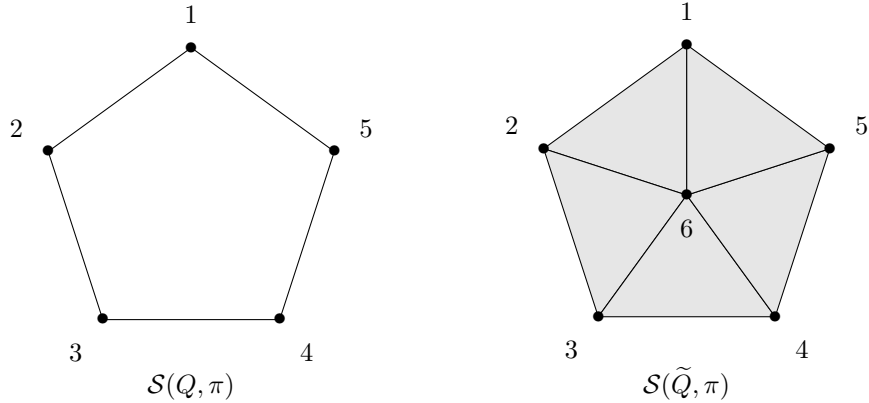
\begin{center}
\begin{tikzpicture}
    \def\r{2}
    \node at (0,-\r-0.5) {$\SC{Q}{\pi}$};
    \coordinate (1) at (90:\r);
    \coordinate (2) at (162:\r);
    \coordinate (3) at (234:\r);
    \coordinate (4) at (306:\r);
    \coordinate (5) at (378:\r);
    \node[label=90:1] at (1) {$\bullet$};
    \node[label=162:2] at (2) {$\bullet$};
    \node[label=234:3] at (3) {$\bullet$};
    \node[label=306:4] at (4) {$\bullet$};
    \node[label=378:5] at (5) {$\bullet$};
    \draw (1) -- (2) -- (3) -- (4) -- (5) -- (1);
\end{tikzpicture}
\qquad
\qquad
\begin{tikzpicture}
    \def\r{2}
    \node at (0,-\r-0.5) {$\SC{\widetilde{Q}}{\pi}$};
    \coordinate (1) at (90:\r);
    \coordinate (2) at (162:\r);
    \coordinate (3) at (234:\r);
    \coordinate (4) at (306:\r);
    \coordinate (5) at (378:\r);
    \coordinate (6) at (0,0);
    \filldraw[draw=black, fill=gray!20] (1) -- (2) -- (6) -- cycle;
    \filldraw[draw=black, fill=gray!20] (2) -- (3) -- (6) -- cycle;
    \filldraw[draw=black, fill=gray!20] (3) -- (4) -- (6) -- cycle;
    \filldraw[draw=black, fill=gray!20] (4) -- (5) -- (6) -- cycle;
    \filldraw[draw=black, fill=gray!20] (5) -- (1) -- (6) -- cycle;
    \node[label=90:1] at (1) {$\bullet$};
    \node[label=162:2] at (2) {$\bullet$};
    \node[label=234:3] at (3) {$\bullet$};
    \node[label=306:4] at (4) {$\bullet$};
    \node[label=378:5] at (5) {$\bullet$};
    \node[label=below:6] at (6) {$\bullet$};
\end{tikzpicture}
\end{center}
\end{example}

Two fundamental theorems in the theory of subword complexes are:

\begin{theorem}[{\cite[Theorem~2.5]{knutson_subword_2004}}]
Subword complexes $\SC{Q}{\pi}$ are vertex decomposable, hence shellable. 
\end{theorem}

\begin{theorem}[{\cite[Theorem~3.7]{knutson_subword_2004}}]
A non-empty subword complex $\SC{Q}{\pi}$ is either a ball or a sphere. Moreover, it is a sphere if and only if the Demazure product $\Dem{Q}=\pi$. 
\end{theorem}

Here, the \defn{Demazure product} $\Dem{Q}\in W$ is an element of the group which is defined recursively as follows. If $Q'$ is the word obtained from $Q$ by adding $s\in S$ at the end, then

\[
  \Dem{Q'} = 
  \begin{cases}
    \mu s, & \text{if } \ell(\mu s)>\ell(\mu), \\
    \mu, & \text{if } \ell(\mu s)<\ell(\mu),
  \end{cases}
\]
where $\mu=\Dem{Q}$. Alternatively, the Demazure product can be defined as follows.

\begin{lemma}{\cite[Lemma~3.4(1)]{knutson_subword_2004}}
\label{lem_Dem_bruhat}
The Demazure product $\Dem{Q}\in W$ is the unique maximal element in Bruhat order among all expressions obtained from subwords of $Q$.
\end{lemma} 


\section{Manifold Reconstruction}\label{sec:reconstruction}
In this section, we revisit Blind and Mani's foundational work on the topic of reconstruction from the facet-ridge graph. We have updated some of their terminology, and generalized some of their lemmas. We use the same numbering of lemmas, for easier cross-reference. In every case, we bold the portion of the lemma that is mathematically different from their work. We make no explicit mention of notation and style changes. Further generalizations of their lemmas may be possible, but are not required for our purposes in this work.

In this section, \(A\) and \(B\) will always refer to some simplicial complexes, with further conditions specified within the lemmas. The function \(f\) will always be an isomorphism from the facet-ridge graph of \(A\) to the facet-ridge graph of \(B\). There is an induced map on faces, \(g : A \to B\), defined as follows:

\[g(I) = \bigcap_{F \supset I} f(F)\]

It is important to note that \(g\) restricted to facets and ridges is identical to \(f\). In general, \(g\) may not be an isomorphism for any other faces.

\addtocounter{BMlemma}{1}
\begin{BMlemma}{\cite[Lemma~2]{blind_puzzles_1987}}
Let \(A,B\) be simplicial \textbf{manifolds without boundary}. 
If \(I\) is a \((d-2)\)-face of \(A\), then either \(g(I)\) is a \((d-2)\)-face, or every \((d-2)\)-face of \(\complex{f(F(star_A(I)))}\) is contained in only one facet.
\end{BMlemma}

\begin{proof}
The proof of this lemma remains identical to the original proof.
\end{proof}

\addtocounter{BMlemma}{1}
\begin{BMlemma}\label{lem_bm_4}{\cite[Lemma~4]{blind_puzzles_1987}}
Let \(A,B\) be \textbf{simplicial spheres}, and let \(I\) be a \((d-2)\)-face of \(A\). 
Then either  \(g(I)\) is a \((d-2)\)-face, or \(\widetilde{H}_{d-2}[B- I]\) is nontrivial, where 
\(B-I = \complex{F(B)\setminus f(F(star_A(I))) }\).
\end{BMlemma}

\begin{proof}
In the original proof, the authors suggest that their work could be extended in this way. 
Using the generalized Schoenflies theorem \cite{morse_1960, brown_1960} fills the gap left when extending the original lemma to non-polytopal spheres.
\end{proof}

\addtocounter{BMlemma}{3}
\begin{BMlemma}\label{lem_bm_8}{\cite[Lemma~8]{blind_puzzles_1987}}
Let \(A\) and \(B\) be \textbf{strongly shellable simplicial spheres}. 
Then \(g\) restricted to the \((d-2)\)-faces of $A$ is a bijection to the \((d-2)\)-faces of $B$.
\end{BMlemma}

\begin{proof}
This follows from the previously updated Lemma, and from the assumption that without loss of generality, \(A\) has at least as many \((d-3)\)-faces as \(B\). 
Of course, following the result of the lemma, \(A\) and \(B\) have the same number of \((d-3)\)-faces, but this is not assumed a priori.
\end{proof}

We conclude this section with a more general version of Blind and Mani's Theorem 1.

\begin{theorem}[{\cite[Theorem~1]{blind_puzzles_1987}}]
\label{thm_BlindMani}
If \(A\) and \(B\) are strongly shellable spheres, and \(f\) is a bijection between the facet-ridge graphs of \(A\) and \(B\), 
then \(g:A\rightarrow B\) is a simplicial isomorphism.
Moreover, it is the unique simplicial isomorphism extension of $f$.
\end{theorem}

\begin{proof}
The proof is identical to Blind and Mani's proof, with the new Lemma~\ref{lem_bm_8}.
\end{proof}

\section{Spherical Subword Complexes}\label{sec:sphereicalSubword}

The main purpose of this section is to show that spherical subword complexes of finite type are strongly vertex decomposable (Theorem~\ref{thm_strongvertexdec_finite}). 
This will be used to prove our main result in Section~\ref{sec:SubwordComplexReconstruction} (Theorem~\ref{thm_SC_reconstruction}). 
We also present a necessary and sufficient condition for the same statement to hold for infinite types (Proposition~\ref{prop_infinite_equivalentconjectures}). We start by recalling some basic properties about subword complexes of finite type. 

\subsection{Basic properties for finite types}
Let $(W,S)$ be a Coxeter system of finite type. 
The first lemma allows us to restrict the study of spherical subword complexes to those for which $\pi=\wo$.

\begin{lemma}\label{lem_finite_one}
Let $\SC{Q}{\pi}$ be a spherical subword complex of finite type. Then $\SC{Q}{\pi}\cong \SC{Q'}{\wo}$ for some word $Q'$.
\end{lemma}
\begin{proof}
Let $Q=(q_1,\dots,q_r)$. By the spherical condition we know that the Demazure product $\Dem{Q}=\pi$. Consider a word $P=(p_1,\dots ,p_k)$ which gives a reduced expression of $\pi^{-1}\wo$, and let $Q'=(q_1,\dots,q_r,p_1,\dots,p_k)$ be the word obtained by appending $P$ to~$Q$. 
Note that every reduced expression of $\pi$ in~$Q$ determines a reduced expression of $\wo$ in $Q'$ by adding all the letters of $(p_1,\dots,p_k)$ to it.
In the other direction, every reduced expression of $\wo$ in $Q'$ contains all the letters of $(p_1,\dots,p_k)$, and its restriction to $Q$ is a reduced expression of $\pi$; otherwise, its restriction to $Q$ would give a reduced expression whose length is bigger than $\ell(\pi)$, which contradicts the fact that $\Dem{Q}=\pi$. 
As a consequence, the facets of $\SC{Q'}{\wo}$ are exactly the facets of $\SC{Q}{\pi}$.
\end{proof}

The subword complexes of the form $\SC{Q}{\wo}$ are especially interesting and have received a lot of attention in the literature.   
The involution $\psi:S\rightarrow S$ given by $\psi(s)=\wo^{-1} s \wo$ allows us to define the \defn{rotated word} $\rotatedword{Q}{s}$ or the \defn{rotation} of $Q = (s,q_2,\ldots,q_r)$ along the letter $s$ as $(q_2, \dots, q_r, \psi(s))$.

\begin{lemma}[{\cite[Proposition~3.9]{ceballos_subword_2014}}]
\label{lem_finite_two}
Let $Q = (s,q_2,\ldots,q_r)$. Then $\SC{Q}{\wo}\cong \SC{\rotatedword{Q}{s}}{\wo}$.
\end{lemma}

The isomorphism between $\SC{Q,\wo}$ and $\SC{\rotatedword{Q}{s}}{\wo}$ in the previous lemma is induced by sending $q_i$ to $q_i$ for $2 \leq i \leq r$ and the initial $s$ to the final $\psi(s)$ (in terms of positions, $i$ is mapped to $i-1$ mod $r$).

\begin{lemma}[c.f.~{\cite[Proof of Theorem~2.5]{knutson_subword_2004}}]
\label{lem_finite_three}
Let $Q=(q_1,q_2,\dots ,q_r)$ and $Q'=(q_2,\dots ,q_r)$. Assume $\{1\}\in \SC{Q}{\pi}$ is a vertex of the subword complex (not necessarily of finite type). Then,
\[
\del{\SC{Q}{\pi}}{1} \cong 
\begin{cases}
    \SC{Q'}{\pi}, & \text{if } \ell(q_1\pi)>\ell(\pi) \\
    \SC{Q'}{q_1\pi}, & \text{if } \ell(q_1\pi)<\ell(\pi) 
\end{cases}
\]
Hence, $\del{\SC{Q}{\pi}}{1}$ is vertex decomposable.
\end{lemma}

\begin{proposition}\label{prop_finite_del_vertexdecompsable}
For every spherical subword complex $\SC{Q}{\pi}$ of finite type and a vertex $\{i\}\in \SC{Q}{\pi}$, the deletion $\del{\SC{Q}{\pi}}{i}$ is vertex decomposable.
\end{proposition}

\begin{proof}
By Lemma~\ref{lem_finite_one} we have $\SC{Q}{\pi}\cong \SC{Q'}{\wo}$ for some word $Q'$, and $\del{\SC{Q}{\pi}}{i}=\del{\SC{Q'}{\wo}}{i}$. Using rotations on $Q'$ and applying Lemma~\ref{lem_finite_two} we can further assume that $i=1$. Finally, by Lemma~\ref{lem_finite_three} we get that $\del{\SC{Q'}{\wo}}{1}$ is vertex decomposable. The result follows.  
\end{proof}

We remark that this proposition does not hold if we drop the spherical condition, see Example~\ref{ex_sub_running_two}.

\subsection{Strong vertex decomposability}\label{subsec:strongvertexdec_finite}
Our main goal is to prove the following result.
\begin{theorem}\label{thm_strongvertexdec_finite}
Spherical subword complexes of finite type are strongly vertex decomposable, hence strongly shellable.
\end{theorem}

By definition, this is equivalent to showing that the star and costar of any non-empty face are vertex decomposable.
These properties are proven in Theorem~\ref{thm_starSC_vertex_decomposable} and Theorem~\ref{thm_ausSC_vertex_decomposable}, respectively. The first is a fairly simple argument, but the second is significantly more involved.

\subsubsection{The star of a face}
\begin{theorem}\label{thm_starSC_vertex_decomposable}
Let $\SC{Q}{\pi}$ be a spherical subword complex and $I\in\SC{Q}{\pi}$ be a non-empty face.
The complex $\star{\SC{Q}{\pi}}{I}$ is vertex decomposable.  
\end{theorem}

\begin{proof}
The star of a face $I$ is equal to the join of $I$ and its link. So, it suffices to show that the link is vertex decomposable. 
The link $\link{\SC{Q}{\pi}}{I}$ can be naturally identified with the subword complex $\SC{Q'}{\pi}$, where $Q'$ is obtained from $Q$ by removing the letters $q_i$ for $i\in I$. Since subword complexes are vertex decomposable, the result follows.   
\end{proof}

\subsubsection{The costar of a face}

\begin{theorem}\label{thm_ausSC_vertex_decomposable}
Let $\SC{Q}{\pi}$ be a spherical subword complex and $I\in\SC{Q}{\pi}$ be a non-empty face.
The complex $\ausSC{Q}{\pi}{I}$ is vertex decomposable. 
\end{theorem}

We remark that the spherical condition in this theorem can not be dropped, as shown in the following example.  

\begin{example}\label{ex_sub_running}
Let $W=S_3$ be the symmetric group generated by simple transpositions $s_i=(i, i+1)$ with $i=1,2$. 
Let $Q=(q_1,q_2,q_3,q_4)=(s_1,s_2,s_1,s_2)$ and $\pi=s_1s_2$. 
The subword complex $\SC{Q}{\pi}$ is a topological ball which is depicted below and whose facets are $\{1,2\},\{2,3\},\{3,4\}$.
\begin{center}
\begin{tikzpicture}
    \node[label=above:1] (1) at (0,0) {$\bullet$};
    \node[label=above:2] (2) at (1,0) {$\bullet$};
    \node[label=above:3] (3) at (2,0) {$\bullet$};
    \node[label=above:4] (4) at (3,0) {$\bullet$};
    \draw (1) -- (2) -- (3) -- (4);
\end{tikzpicture}
\end{center}
Taking $I=\{2,3\}$, the complex $\ausSC{Q}{\pi}{I}$ has two facets $\{1,2\}$ and $\{3,4\}$:
\begin{center}
\begin{tikzpicture}
    \node[label=above:1] (1) at (0,0) {$\bullet$};
    \node[label=above:2] (2) at (1,0) {$\bullet$};
    \node[label=above:3] (3) at (2,0) {$\bullet$};
    \node[label=above:4] (4) at (3,0) {$\bullet$};
    \draw (1) -- (2);
    \draw (3) -- (4);
\end{tikzpicture}
\end{center}
This complex is not even shellable, thus not vertex decomposable. 
\end{example}

As a warm up exercise, we start by proving the special case $|I|=1$ of Theorem~\ref{thm_ausSC_vertex_decomposable} in Corollary~\ref{cor_case_I_1}, for which we use the following lemma. 

\begin{lemma}\label{lem_case_I_1_one}
Let $\SC{Q}{\pi}$ be a spherical subword complex and $I=\{i\}\in\SC{Q}{\pi}$. Then
\[
\ausSC{Q}{\pi}{I} = \del{\SC{Q}{\pi}}{i}.
\]
\end{lemma}
\begin{proof}
We will prove containment in both directions. 

\smallskip
Proof of $\subseteq$:
Let $J\in \ausSC{Q}{\pi}{I}$. 
By definition, this face $J$ is contained in a facet $\widetilde J \in \SC{Q}{\pi}$ such that $I=\{i\} \nsubseteq \widetilde J$. Therefore, $\widetilde J \in \del{\SC{Q}{\pi}}{i}$. Since $J \subseteq \widetilde J$, then $J \in \del{\SC{Q}{\pi}}{i}$.

\medskip
Proof of $\supseteq$:
Let $J\in \del{\SC{Q}{\pi}}{i}$. 
This means $J \in \SC{Q}{\pi}$ such that $\{i\} \nsubseteq J$.
Consider a facet $\widetilde J \in \SC{Q}{\pi}$ such that $J \subseteq \widetilde J$.
We have two possible cases:
\begin{itemize}
    \item  If $I=\{i\}\nsubseteq \widetilde J$ then $\widetilde J \in \ausSC{Q}{\pi}{I}$. 
Since $J \subseteq \widetilde J$ then $J \in \ausSC{Q}{\pi}{I}$. 
    \item If $I=\{i\}\subseteq \widetilde J$, then we can use the spherical condition and flip $i$ to get a new facet $\widetilde J'=\widetilde J \setminus i \cup i' \in \SC{Q}{\pi}$.
    Since $J \subseteq \widetilde J$ and $i\notin J$, then $J \subseteq \widetilde J'$. 
    Applying the previous case to the facet $\widetilde J'$ we deduce $J \in \ausSC{Q}{\pi}{I}$.
\end{itemize}
In both cases we have proven $J \in \ausSC{Q}{\pi}{I}$ as desired.
\end{proof}



Lemma~\ref{lem_case_I_1_one} does not necessarily hold when the subword complex is not spherical. Here is a counter-example: 
\begin{example}[Example~\ref{ex_sub_running} continued] \label{ex_sub_running_two}
Let $W=S_3$, the word $Q=(q_1,q_2,q_3,q_4)=(s_1,s_2,s_1,s_2)$ and $\pi=s_1s_2$. 
Taking $I=\{2\}$ we get that $\ausSC{Q}{\pi}{I}$ has exactly one facet $\{3,4\}$:
\begin{center}
\begin{tikzpicture}
    \node at (-1.5,0) {$\ausSC{Q}{\pi}{I}:$};
    \node[label=above:] (1) at (0,0) {$\circ$};
    \node[label=above:] (2) at (1,0) {$\circ$};
    \node[label=above:3] (3) at (2,0) {$\bullet$};
    \node[label=above:4] (4) at (3,0) {$\bullet$};
    \draw (3) -- (4);
\end{tikzpicture}
\end{center}
On the other hand, $\del{\SC{Q}{\pi}}{2}$ has two facets $\{1\},\{3,4\}$:
\begin{center}
\begin{tikzpicture}
    \node at (-1.7,0) {$\del{\SC{Q}{\pi}}{2}:$};
    \node[label=above:1] (1) at (0,0) {$\bullet$};
    \node[label=above:] (2) at (1,0) {$\circ$};
    \node[label=above:3] (3) at (2,0) {$\bullet$};
    \node[label=above:4] (4) at (3,0) {$\bullet$};
    \draw (3) -- (4);
\end{tikzpicture}
\end{center}
Lemma~\ref{lem_case_I_1_one} does not hold in this case because 
$\ausSC{Q}{\pi}{I} \neq \del{\SC{Q}{\pi}}{i}$ for $i=2$.
Proposition~\ref{prop_finite_del_vertexdecompsable} does not hold either because $\del{\SC{Q}{\pi}}{2}$ is not vertex decomposable (it is not even pure).
\end{example}

\begin{corollary}\label{cor_case_I_1}
Let $\SC{Q}{\pi}$ be a spherical subword complex and $I=\{i\}\in\SC{Q}{\pi}$.
The complex $\ausSC{Q}{\pi}{I}$ is vertex decomposable. 
\end{corollary}
\begin{proof}
By Lemma~\ref{lem_case_I_1_one}, we get 
$\ausSC{Q}{\pi}{I} = \del{\SC{Q}{\pi}}{i}$.
By Proposition~\ref{prop_finite_del_vertexdecompsable} $\del{\SC{Q}{\pi}}{i}$ is vertex decomposable. 
\end{proof}

This corollary shows Theorem~\ref{thm_ausSC_vertex_decomposable} for the special case $|I|=1$. For the case~$|I|\geq 2$ we need to analyze the behaviour of the complex $\ausSC{Q}{\pi}{I}$ under the link and deletion operations. We will perform these operations on vertices of the complex that belong to the set $I$. The following lemma shows that such elements $i\in I$ are indeed vertices of $\ausSC{Q}{\pi}{I}$ when $|I|\geq 2$. 

\begin{lemma}\label{lem_ausSC_I_vertices}
Let $\SC{Q}{\pi}$ be a spherical subword complex and $I\in\SC{Q}{\pi}$ be a face  with $|I|\geq 2$.
For every $i\in I$ we have that $\{i\}\in \ausSC{Q}{\pi}{I}$. 
\end{lemma}

The condition $|I|\geq 2$ is necessary because otherwise $I=\{i\}\notin \ausSC{Q}{\pi}{I}$.
This is one reason why we treated the case $I=\{i\}$ separetely in Corollary~\ref{cor_case_I_1}.
The spherical condition is also critical, as shown in the following example. 
\begin{example}\label{ex_sub_running_three}
[Example~\ref{ex_sub_running} continued] 
Let $W=S_3$, $Q=(q_1,q_2,q_3,q_4)=(s_1,s_2,s_1,s_2)$ and $\pi=s_1s_2$ as above. If we take the facet $I=\{1,2\}$, then $\ausSC{Q}{\pi}{I}$ has two facets $\{2,3\},\{3,4\}$:
\begin{center}
\begin{tikzpicture}
    \node[label=above:] (1) at (0,0) {$\circ$};
    \node[label=above:2] (2) at (1,0) {$\bullet$};
    \node[label=above:3] (3) at (2,0) {$\bullet$};
    \node[label=above:4] (4) at (3,0) {$\bullet$};
    \draw (2) -- (3);
    \draw (3) -- (4);
\end{tikzpicture}
\end{center}
The element $1\in I$ but $\{1\}\notin \ausSC{Q}{\pi}{I}$.
\end{example}

\begin{proof}[Proof of Lemma~\ref{lem_ausSC_I_vertices}]
Let $i\neq i'$ be two different elements in $I$. Let $J\in \SC{Q}{\pi}$ be a facet such that $I\subseteq J$. Since the subword complex is spherical, we can flip the element $i'\in J$ to get a new facet $J'=J\setminus i'\cup j\in \SC{Q}{\pi}$, for some $j\neq i'$. Clearly $I\nsubseteq J'$, and so $J'\in \ausSC{Q}{\pi}{I}$. Since $\{i\}\subseteq J'$ then $\{i\}\in \ausSC{Q}{\pi}{I}$. 
\end{proof}

Now that we know that every element $i\in I$ is a vertex of $\ausSC{Q}{\pi}{I}$ (when $|I|\geq 2$), we can proceed performing the link and deletion operations on these vertices. 

To simplify notation, we denote by $Q\setminus q_i=(q_1,\dots,\widehat{q_i},\dots,q_r)$ the word obtained by deleting the letter $q_i$ from $Q$. We consider the subword complex $\SC{Q\setminus q_i}{\pi}$ as a simplicial complex on $[r]\setminus i$. In particular, if $I\in \SC{Q}{\pi}$ is a face with $i\in I$, then $I\setminus i$ is a face of $ \SC{Q\setminus q_i}{\pi}$.


\begin{lemma}\label{lem_link}
Let $\SC{Q}{\pi}$ be a spherical subword complex and $I\in\SC{Q}{\pi}$ be a face  with $|I|\geq 2$.
For every $i\in I$ we have
\[
\link{\ausSC{Q}{\pi}{I}}{i} = 
\ausSC{Q\setminus q_i}{\pi}{I\setminus i}.
\]
\end{lemma}

\begin{proof}
By Lemma~\ref{lem_ausSC_I_vertices} we have that  $\{i\}\in \ausSC{Q}{\pi}{I}$, and we can compute its link.
We prove the desired equality by showing the containment in both directions:

\medskip
Proof of $\subseteq$:
Let $J\in \link{\ausSC{Q}{\pi}{I}}{i}$. This means that $J\cup i\in \ausSC{Q}{\pi}{I}$ and $i\notin J$. By definition of $\ausSC{Q}{\pi}{I}$, there exists a facet $\widetilde J \cup i \in \SC{Q}{\pi}$ such that $J\cup i \subseteq \widetilde J \cup i$ and $I \nsubseteq \widetilde J \cup i$. Here we assume $i \notin \widetilde J$. 
In particular, $\widetilde J$ is a facet of $\SC{Q\setminus q_i}{\pi}$ and $I\setminus i \nsubseteq \widetilde J$. Therefore, $\widetilde J\in \ausSC{Q\setminus q_i}{\pi}{I\setminus i}$. Since $J \subseteq \widetilde J$ then $J\in \ausSC{Q\setminus q_i}{\pi}{I\setminus i}$.

\medskip
Proof of $\supseteq$:
Let $J\in \ausSC{Q\setminus q_i}{\pi}{I\setminus i}$.
This means that there exists a facet $\widetilde J\in \SC{Q\setminus q_i}{\pi}$ such that $J\subseteq \widetilde J$ and $I\setminus i \nsubseteq \widetilde J$.
Therefore $\widetilde J \cup i \in \SC{Q}{\pi}$ is a facet such that $I \nsubseteq \widetilde J \cup i$. 
This is equivalent to $\widetilde J \cup i \in \ausSC{Q}{\pi}{I}$.
Thus, $\widetilde J\in \link{\ausSC{Q}{\pi}{I}}{i}$. 
Since $J \subseteq \widetilde J$, we also have $J\in \link{\ausSC{Q}{\pi}{I}}{i}$. 
\end{proof}

We remark that the spherical condition was not completely necessary in the proof of the previous lemma. All we used is that $\{i\}\in \ausSC{Q}{\pi}{I}$ in order to be able to compute the link.

\begin{lemma}\label{lem_deletion}
Let $\SC{Q}{\pi}$ be a spherical subword complex and $I\in\SC{Q}{\pi}$ be a face  with $|I|\geq 2$.
For every $i\in I$ we have
\[
\del{\ausSC{Q}{\pi}{I}}{i} = 
\del{\SC{Q}{\pi}}{i}.
\]
\end{lemma}
\begin{proof}
By Lemma~\ref{lem_ausSC_I_vertices} we have that  $\{i\}\in \ausSC{Q}{\pi}{I}$, and we can compute its deletion.
We prove the desired equality by showing the containment in both directions:

\medskip
Proof of $\subseteq$:
Let $J\in \del{\ausSC{Q}{\pi}{I}}{i}$. 
Then, $J\in \ausSC{Q}{\pi}{I}$ with $i \notin J$.
Since every face of $\ausSC{Q}{\pi}{I}$ is a face of $\SC{Q}{\pi}$, 
we have that $J\in \SC{Q}{\pi}$ with $i \notin J$.
This implies that $J \in \del{\SC{Q}{\pi}}{i}$.

\medskip
Proof of $\supseteq$: this is the tricky part where the spherical condition is useful. 
Let $J\in \del{\SC{Q}{\pi}}{i}$.
This is equivalent to $J\in \SC{Q}{\pi}$ with $i \notin J$.
Let $\widetilde J\in \SC{Q}{\pi}$ be a facet such that $J\subseteq \widetilde J$.
We have two possibilities:
\begin{itemize}
    \item If $i\notin \widetilde J$, then $\widetilde J \in \ausSC{Q}{\pi}{I}$ and so $J \in \ausSC{Q}{\pi}{I}$. 
    Since $i\notin J$, then $J\in \del{\ausSC{Q}{\pi}{I}}{i}$.
    \item If $i \in \widetilde J$, then flipping $i$ creates a new facet $\widetilde J'\in \SC{Q}{\pi}$. This facet satisfies $J \subset \widetilde J'$ and $i \notin \widetilde J'$. Using the previous item for this facet, we deduce $J\in \del{\ausSC{Q}{\pi}{I}}{i}$.
\end{itemize}
In both cases we have proved $J\in \del{\ausSC{Q}{\pi}{I}}{i}$ as desired.
\end{proof}

This lemma does not hold for non-spherical subword complexes. Here is a counter-example:

\begin{example}\label{ex_sub_running_four}
Let $W=S_4$ be the symmetric group generated by simple transpositions $s_i=(i, i+1)$ for $i=1,2,3$. Let $Q=(s_1,s_2,s_3,s_1,s_2)$ and $\pi=s_1s_2$. 
The subword complex $\SC{Q}{\pi}$ has three facets $\{1,2,3\},\{2,3,4\},\{3,4,5\}$, and is depicted below.
Taking $I=\{2,3\}$, the complex $\ausSC{Q}{\pi}{I}$ has one facet $\{3,4,5\}$:
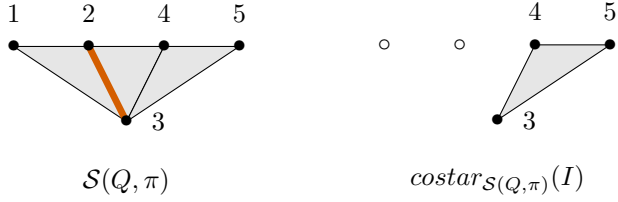
\begin{center}
\begin{tikzpicture}
    \node at (1.5,-1.8) {$\SC{Q}{\pi}$};
    \coordinate (1) at (0,0);
    \coordinate (2) at (1,0);
    \coordinate (4) at (2,0);
    \coordinate (5) at (3,0);
    \coordinate (3) at (1.5,-1);
    \filldraw[draw=black, fill=gray!20] (1) -- (2) -- (3) -- cycle;
    \filldraw[draw=black, fill=gray!20] (2) -- (3) -- (4) -- cycle;
    \filldraw[draw=black, fill=gray!20] (3) -- (4) -- (5) -- cycle;
    \draw[red, line width=0.9mm] (2) -- (3);
    \node[label=above:1] at (1) {$\bullet$};
    \node[label=above:2] at (2) {$\bullet$};
    \node[label=above:4] at (4) {$\bullet$};
    \node[label=above:5] at (5) {$\bullet$};
    \node[label=right:3] at (3) {$\bullet$};
\end{tikzpicture}
\qquad
\qquad
\begin{tikzpicture}
    \node at (1.5,-1.8) {$\ausSC{Q}{\pi}{I}$};
    \coordinate (1) at (0,0);
    \coordinate (2) at (1,0);
    \coordinate (4) at (2,0);
    \coordinate (5) at (3,0);
    \coordinate (3) at (1.5,-1);
    \filldraw[draw=black, fill=gray!20] (3) -- (4) -- (5) -- cycle;
    \node[label=above:] at (1) {$\circ$};
    \node[label=above:] at (2) {$\circ$};
    \node[label=above:4] at (4) {$\bullet$};
    \node[label=above:5] at (5) {$\bullet$};
    \node[label=right:3] at (3) {$\bullet$};
\end{tikzpicture}
\end{center}
For $i=3$, The deletion $\del{\SC{Q}{\pi}}{i}$ has three facets $\{1,2\},\{2,4\},\{4,5\}$, while the deletion $\del{\ausSC{Q}{\pi}{I}}{i}$ has only one facet $\{4,5\}$:

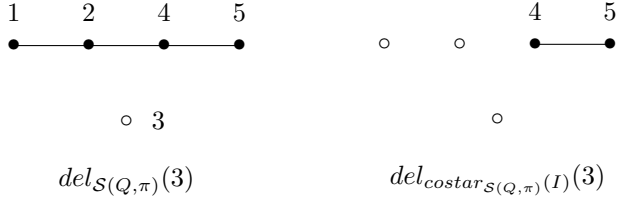
\begin{center}
\begin{tikzpicture}
    \node at (1.5,-1.8) {$\del{\SC{Q}{\pi}}{3}$};
    \coordinate (1) at (0,0);
    \coordinate (2) at (1,0);
    \coordinate (4) at (2,0);
    \coordinate (5) at (3,0);
    \coordinate (3) at (1.5,-1);
    \draw (1) -- (2) -- (4) -- (5);
    \node[label=above:1] at (1) {$\bullet$};
    \node[label=above:2] at (2) {$\bullet$};
    \node[label=above:4] at (4) {$\bullet$};
    \node[label=above:5] at (5) {$\bullet$};
    \node[label=right:3] at (3) {$\circ$};
\end{tikzpicture}
\qquad
\qquad
\begin{tikzpicture}
    \node at (1.5,-1.8) {$\del{\ausSC{Q}{\pi}{I}}{3}$};
    \coordinate (1) at (0,0);
    \coordinate (2) at (1,0);
    \coordinate (4) at (2,0);
    \coordinate (5) at (3,0);
    \coordinate (3) at (1.5,-1);
    \draw (4) -- (5);
    \node[label=above:] at (1) {$\circ$};
    \node[label=above:] at (2) {$\circ$};
    \node[label=above:4] at (4) {$\bullet$};
    \node[label=above:5] at (5) {$\bullet$};
    \node[label=right:] at (3) {$\circ$};
\end{tikzpicture}
\end{center}
These two complexes are not the same.
\end{example}

With these preliminaries, we are now ready to prove Theorem~\ref{thm_ausSC_vertex_decomposable}.

\begin{proof}[Proof of Theorem~\ref{thm_ausSC_vertex_decomposable}]
Let $\SC{Q}{\pi}$ be a spherical subword complex and $I\in\SC{Q}{\pi}$ be a non-empty face.
We want to show that the complex $\ausSC{Q}{\pi}{I}$ is vertex decomposable. 
We have already proved this result in the case when $|I|=1$ in Corollary~\ref{cor_case_I_1}. 
We will prove the general case by induction on the length of~$Q$ and $|I|$.  

Assume $|I|\geq 2$. 
The complex $\ausSC{Q}{\pi}{I}$ is pure by definition. We will show that it has a vertex whose link and deletion are vertex decomposable.  

By Lemma~\ref{lem_ausSC_I_vertices}, any element $i\in I$ is a vertex of $\ausSC{Q}{\pi}{I}$.
By Lemma~\ref{lem_link}
\[
\link{\ausSC{Q}{\pi}{I}}{i} = 
\ausSC{Q\setminus q_i}{\pi}{I\setminus i}.
\]
The next key point is that $\SC{Q\setminus q_i}{\pi}$ is a spherical subword complex. Indeed this is equivalent to show that the Demazure product $\Dem{Q\setminus q_i}=\pi$. This is deduced from the following two inequalities (which are deduced from Lemma~\ref{lem_Dem_bruhat}):
\begin{itemize}
    \item $\Dem{Q\setminus q_i} \leq \Dem{Q}=\pi$,
    \item $\Dem{Q\setminus q_i} \geq \pi$ (which holds because $\{i\}\in\SC{Q}{\pi}$ is a vertex).
\end{itemize}

By induction hypothesis, $\ausSC{Q\setminus q_i}{\pi}{I\setminus i}$ is vertex decomposable, and so $\link{\ausSC{Q}{\pi}{I}}{i}$ is vertex decomposable.

By Lemma~\ref{lem_deletion} we get
\[
\del{\ausSC{Q}{\pi}{I}}{i} 
= \del{\SC{Q}{\pi}}{i}
\]
which is vertex decomposable by Proposition~\ref{prop_finite_del_vertexdecompsable}.
\end{proof}

\subsection{Spherical subword complexes of infinite type}

In Section~\ref{subsec:strongvertexdec_finite}, we showed that spherical subword complexes of finite type are strongly vertex decomposable (Theorem~\ref{thm_strongvertexdec_finite}).  
Except for Proposition~\ref{prop_finite_del_vertexdecompsable}, all the ingredients in the proof of this theorem are valid for spherical subword complexes of infinite type.  
We conjecture that Proposition~\ref{prop_finite_del_vertexdecompsable} holds for infinite types as well. 

\begin{conjecture}\label{conj_infinite_one}
For every spherical subword complex $\SC{Q}{\pi}$ of infinite type and a vertex $\{i\}\in \SC{Q}{\pi}$, the deletion $\del{\SC{Q}{\pi}}{i}$ is vertex decomposable.
\end{conjecture}

This property would imply that spherical subword complexes are strongly vertex decomposable in general.

\begin{conjecture}\label{conj_infinite_two}
Spherical subword complexes of infinite type are strongly vertex decomposable.
\end{conjecture}

In fact, the two conjectures are equivalent.

\begin{proposition}\label{prop_infinite_equivalentconjectures}
Conjecture~\ref{conj_infinite_one} and
Conjecture~\ref{conj_infinite_two} are equivalent.
\end{proposition}

\begin{proof}
Let $\SC{Q}{\pi}$ be a spherical subword complex of infinite type. 

Assume that Conjecture~\ref{conj_infinite_two} holds, and consider a vertex $\{i\}\in \SC{Q}{\pi}$. 
Taking $I=\{i\}$ and applying Lemma~\ref{lem_case_I_1_one} we get 
\[
\ausSC{Q}{\pi}{I} = \del{\SC{Q}{\pi}}{i}.
\]
Since $\ausSC{Q}{\pi}{I}$ is vertex decomposable by assumption, then $\del{\SC{Q}{\pi}}{i}$ is vertex decomposable and Conjecture~\ref{conj_infinite_one} holds.

Now assume that Conjecture~\ref{conj_infinite_one} holds.
The desired Conjecture~\ref{conj_infinite_two} is equivalent to Theorem~\ref{thm_starSC_vertex_decomposable} (which already works for infinite types) and the infinite type version of Theorem~\ref{thm_ausSC_vertex_decomposable}. 
The proof of this modified version is exactly the same as the proof of the original one, replacing Proposition~\ref{prop_finite_del_vertexdecompsable} by Conjecture~\ref{conj_infinite_two} whenever used. 
\end{proof}

\section{Subword complex reconstruction}\label{sec:SubwordComplexReconstruction}

As a consequence of Theorem~\ref{thm_strongvertexdec_finite} and Blind and Mani's original techniques summarized in Theorem~\ref{thm_BlindMani}, we obtain the following result.

\begin{theorem}\label{thm_SC_reconstruction}
Let \(A,B\) be spherical subword complexes of finite type. Every isomorphism between their facet-ridge graphs \(f: FR(A) \to FR(B)\) has a unique extension to a simplicial isomorphism \(g: A \to B\).
\end{theorem}

In other words, spherical subword complexes satisfy Kalai's Conjecture~\ref{conj_kalai}, and are completely determined by their facet-ridge graph. 
However, our proof is not constructive and it is an open problem to describe the faces of the complex in terms of their facet-ridge graph.

\begin{openproblem}
Given the facet-ridge graph of a spherical subword complex, find a direct description of its faces.
\end{openproblem}

We remark that Theorem~\ref{thm_SC_reconstruction} does not hold for subword complexes that are topological balls. 
For instance, the subword complexes in Example~\ref{ex_sub_running} and Example~\ref{ex_sub_running_four} have the same facet-ridge graph but are not isomorphic.

We conjecture that Theorem~\ref{thm_SC_reconstruction} holds in general, not only for finite types. 

\begin{conjecture}\label{conj_infinite_three}
Spherical subword complexes of infinite type are completely determined by their facet-ridge graph.
\end{conjecture}

This conjecture follows from either of the two equivalent Conjecture~\ref{conj_infinite_one} or Conjecture~\ref{conj_infinite_two}.

\section{Non reconstructible manifolds}\label{sec:nonreconstructible}
Not all manifolds are reconstructible. We fully illustrate two examples. The first is the projective plane, the second is a torus.
In each case, we provide a pair of facet lists with the same facet-ridge graph, but different combinatorics. We further illustrate where each component of the Blind and Mani result fails to hold.

\begin{strategy}\label{strategy}
The strategy we use to construct two non-isomorphic simplicial complexes with the same facet-ridge graph is as follows. 
We start with two identical copies $A$ and $A'$ of a triangulation whose simplicial automorphism group has order smaller than the order of the automorphism group of its facet-ridge graph. 
We choose a graph automorphism $f$ which does not come from a simplicial automorphism.
Then, we use $f$ to perform a sequence of stellar subdivisions on $A$ and $A'$ simultaneously, such that at each step the corresponding facet-ridge graphs are isomorphic. 
After performing several of these transformations the combinatorics of the two triangulations will be changed, but there will still be an isomorphism of their facet-ridge graphs.

This is how we proceed. 
If we perform a stellar subdivision of a facet $F$ in $A$, then we simultaneously perform a stellar subdivision of the facet $f(F)$ in $A'$. 
In the graph, this operation replaces a vertex (associated to the facets $F$ and $f(F)$) by the graph of a $d$-simplex, and the adjacent $d+1$ edges to these vertices become adjacent to the $d+1$ vertices of the added simplex. 
We call this operation the \defn{truncation of a vertex}, see Figure~\ref{fig_stellarsubdivision_vertextruncation}.

The resulting two graphs are still isomorphic, and the automorphism relating them is just a small modification of $f$. 
Performing these operations along a well-chosen sequence of facets will eventually cause the combinatorics of the two triangulations to differ from each other.
\end{strategy}

\begin{figure}
    \centering
\begin{tikzpicture}[scale=0.95,every node/.append style={circle,draw=black,fill=white,inner sep=1pt,align=center}]
\node at (210:1) (1){};
\node at (330:1) (2){};
\node at (90:1) (3){};
\draw (1) -- (2) -- (3) -- (1);
\end{tikzpicture}
    \quad
\begin{tikzpicture}[scale=0.95,every node/.append style={circle,draw=black,fill=white,inner sep=1pt,align=center}]
\node at (210:1) (1){};
\node at (330:1) (2){};
\node at (90:1) (3){};
\node at (0,0) (4){};
\draw (1) -- (2) -- (3) -- (1);
\draw (1) -- (4);
\draw (2) -- (4);
\draw (3) -- (4);
\end{tikzpicture}
    \qquad
    \qquad
\begin{tikzpicture}[scale=0.45,every node/.append style={circle,draw=black,fill=white,inner sep=1pt,align=center}]
\node at (0,0) (0){};
\coordinate (1a) at (30:2);
\coordinate (2a) at (150:2);
\coordinate (3a) at (270:2);
\draw (0) -- (1a);
\draw (0) -- (2a);
\draw (0) -- (3a);
\end{tikzpicture}
    \quad
\begin{tikzpicture}[scale=0.45,every node/.append style={circle,draw=black,fill=white,inner sep=1pt,align=center}]
\node at (30:0.7) (1){};
\node at (150:0.7) (2){};
\node at (270:0.7) (3){};
\draw (1) -- (2) -- (3) -- (1);
\coordinate (1a) at (30:2);
\coordinate (2a) at (150:2);
\coordinate (3a) at (270:2);
\draw (1) -- (1a);
\draw (2) -- (2a);
\draw (3) -- (3a);
\end{tikzpicture}
\caption{Stellar subdivision of a 2-dimensional simplex and the truncation of the corresponding vertex on the facet-ridge graph.}
\label{fig_stellarsubdivision_vertextruncation}
\end{figure}
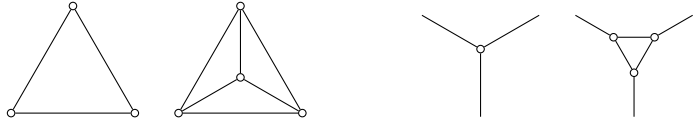

This idea of using graphs with large automorphism group is not new. 
Blind and Mani already mention at the end of their introduction in~\cite{blind_puzzles_1987}, ``also, there exist simple cellular decompositions \(\mathfrak{C}\) of the torus, whose 1-skeleton has a larger automorphism group than \(\mathfrak{C}\) itself.'' 
Adiprisito described a decoration of the minimal triangulation of both the torus and the projective plane as counterexamples to Kalai's conjecture for surfaces~\cite{adiprasito_2017}.
Mohar and Vodopivec found multiple embedings of a certain class of graph in non-orientable surfaces of high genus. 
The dual simplical complexes to these embeddings forms a large class of complexes with the same facet-ridge graph but different combinatorics~\cite[Section 5]{mohar_vodopivec_2006}.

We offer the rest of this section as a concrete and fully detailed expansion of these ideas. 
With facet lists explicitly written and isomorphisms clearly defined, we hope these examples will provide in-context insights into the obstruction to extending Kalai's conjecture to more general surfaces.

\subsection{Projective plane}
We begin this construction with the well-known minimal triangulation of the projective plane. 
Figure \ref{fig:min_rp2} represents this triangulation in such a way as to emphasize the large degree of symmetry, which is the source of the large automorphism group of this simplicial complex and its facet-ridge graph.

\begin{figure}
    \centering
    \begin{tikzpicture}[scale=0.95,every node/.append style={circle,draw=black,fill=white,inner sep=1pt,align=center}]
    \node[white] at (90:3.1){};
    \node[white] at (0:3.1){};
    \node[white] at (180:3.1){};
    \node[white] at (270:3.1){};
\node at (210:1) (1){5};
\node at (330:1) (2){4};
\node at (90:1) (3){6};
\node at (90:2) (4a){1};
\node at (150:2) (5a){3};
\node at (210:2) (6a){2};
\node at (270:2) (4b){1};
\node at (330:2) (5b){3};
\node at (30:2) (6b){2};
\draw (1) -- (2) -- (3) -- (1);
\draw (6a) -- (5a) -- (4a) -- (6b) -- (5b) -- (4b) -- (6a);
\draw (1) --(4b) -- (2) -- (6b) -- (3) -- (5a) -- (1);
\draw (1) -- (6a);
\draw (3) -- (4a);
\draw (2) -- (5b);

\draw[bluishgreen] (90:0.4) arc (-90:270:0.6);
\draw[bluishgreen] (0:{sqrt(3)}) arc (49.1:130.9:{sqrt(7)});
    
    \end{tikzpicture}
    \hspace{1pt}
    \begin{tikzpicture}[scale=0.95,every node/.append style={circle,draw=black,fill=white,inner sep=1pt,align=center}]
\pgfmathsetmacro{\s}{15};
    \node[white] at (90:3.1){};
    \node[white] at (0:3.1){};
    \node[white] at (180:3.1){};
    \node[white] at (270:3.1){};
    \node at (0:0) (456){456};
    \node at (150:1) (356){356};
    \node at (30:1) (246){246};
    \node at (270:1) (145){145};
    \node at (240:2) (125){125};
    \node at (180:2) (235){235};
    \node at (120:2) (136){136};
    \node at (60:2) (126){126};
    \node at (0:2) (234){234};
    \node at (300:2) (134){134};
    \draw (456) -- (356) -- (235) -- (125) -- (145) -- (456) -- (246) -- (126) -- (136) -- (356);
    \draw (246) -- (234) -- (134) -- (145);
    
\draw (234) to[out=330-\s, in=0] (270:3) to[in=210+\s,out=180] (235);
\draw (136) to[out=90-\s, in=120] (30:3) to[in=330+\s,out=300] (134);
\draw (125) to[out=210-\s, in=240] (150:3) to[in=90+\s,out=60] (126);

\draw[->,red,thick,dashed] (456) to[in=0, out=60] (90:0.5) to[in=120, out=180] (456);
\draw[->,red,thick,dashed] (356) to[in=0, out=300] (180:0.88) to[in=240, out=180] (356);
\draw[->,red,thick,dashed] (145) to[in=120, out=60] (300:0.88) to[in=0, out=300] (145);
\draw[->,red,thick,dashed] (246) to[in=240, out=180] (60:0.88) to[in=120, out=60] (246);
\draw[<->,red,thick,dashed] (235) -- (136);
\draw[<->,red,thick,dashed] (234) -- (126);
\draw[<->,red,thick,dashed] (125) -- (134);
    \end{tikzpicture}
    \caption{Minimal Triangulation of the Projective plane (left), and its facet-ridge graph together with a graph automorphism which does not correspond to a simplicial automorphism (right).}
    \label{fig:min_rp2}
\end{figure}
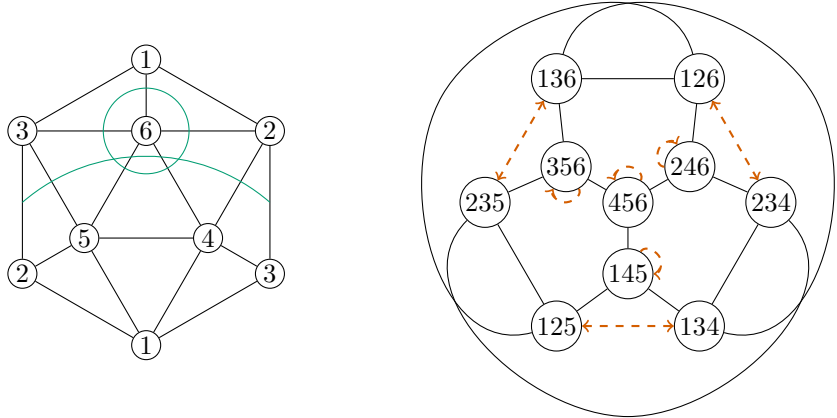

\[
\begin{tabular}{cccccc}
    \{1,2,5\}, & \{1,4,5\}, & \{1,3,4\}, & \{2,3,4\}, & \{2,4,6\}, & \{1,2,6\}, \\ \relax
    \{1,3,6\}, & \{3,5,6\}, & \{2,3,5\}, & \{4,5,6\} \\
\end{tabular}
\]

The automorphism group of this simplicial complex is of order 60, while the automorphism gorup of its facet-ridge graph is of order 120.
We choose a graph automorphism which does not correspond to a simplicial automorphism. 
This chosen facet-ridge graph isomorphism $f$ is the involution illustrated in Figure \ref{fig:min_rp2} (right).
Under this involution the two {\color{bluishgreen}green} paths are transformed into each other.
The map $f$ does not extend to a simplicial isomorphism. 
Indeed, such an extension $g$ would be uniquely determined by 
\[g(I) = \bigcap_{F \supset I} f(F)\]
 where $F$ runs over all facets containing $I$. But $g(\{6\})=\emptyset$, because it is the intersection of $f$ applied to the five triangles along the {\color{bluishgreen}green} path surrounding the vertex $\{6\}$; these are the five triangles along the other {\color{bluishgreen}green} path, and their intersection is empty. 
 
Now, we use this graph automorphism to create two non-isomorphic complexes by applying a sequences of stellar subdivisions as explained in Strategy~\ref{strategy}. We stellarly subdivide \(\{4,5,6\}\) with \(\{7\}\) and \(\{2,4,6\}\) with \(\{8\}\). 
We then subdivide \(\{4,5,7\}\) with \(\{9\}\). 
To make the different complexes, we subdivide either \(\{2,4,8\}\) with \(\{a\}\) or \(\{2,6,8\}\) with \(\{a\}\). 
The resulting complex is shown in Figure \ref{fig_projectiveplane_two_nonismorphic_triangulations}, with the two subdivisions shown in different colors. 
The facet lists of these two complexes follow.

\begin{figure}
\begin{center}
\begin{tikzpicture}[scale=1.7,every node/.append style={circle,draw=black,fill=white,inner sep=1pt,align=center}]
\node[draw=red] at (45:1.2) (ab){a};
\node[draw=blue] at (15:1.2) (aa){a};
\node at (0:0) (7){7};
\node at (30:1.1) (8){8};
\node at (270:0.3) (9){9};
\node at (210:1) (1){5};
\node at (330:1) (2){4};
\node at (90:1) (3){6};
\node at (90:2) (4a){1};
\node at (150:2) (5a){3};
\node at (210:2) (6a){2};
\node at (270:2) (4b){1};
\node at (330:2) (5b){3};
\node at (30:2) (6b){2};
\draw (1) -- (2) -- (3) -- (1);
\draw (6a) -- (5a) -- (4a) -- (6b) -- (5b) -- (4b) -- (6a);
\draw (1) --(4b) -- (2) -- (6b) -- (3) -- (5a) -- (1);
\draw (1) -- (6a);
\draw (3) -- (4a);
\draw (2) -- (5b);
\draw (7) -- (1);
\draw (2) -- (7) -- (3);
\draw (9) -- (7);
\draw (1) -- (9) -- (2);
\draw (2) -- (8) -- (3);
\draw (6b) -- (8);
\draw[color=red] (3) -- (ab) -- (6b);
\draw[color=red] (8) -- (ab);
\draw[color=blue] (2) -- (aa) -- (6b);
\draw[color=blue] (8) -- (aa);
\end{tikzpicture}
\end{center}
    \caption{Two non-isomorphic triangulations of the projective plane with isomorphic facet-ridge graphs.}
    \label{fig_projectiveplane_two_nonismorphic_triangulations}
\end{figure}
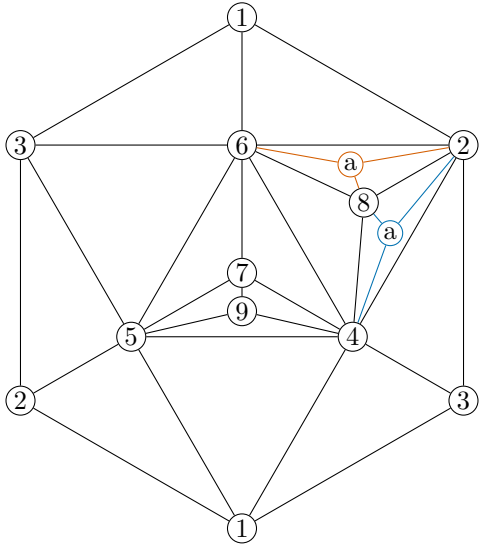

\[
\begin{tabular}{cccccc}
    \{1,3,4\}, & \{1,4,5\}, & \{1,2,5\}, & \{1,2,6\}, &          & \{2,3,4\}, \\ \relax
    \{2,3,5\}, & \{3,5,6\}, & \{1,3,6\}, &  \\ \relax
               & \{4,7,6\}, & \{7,5,6\}, & \{2,4,8\}, &          & \{8,4,6\}, \\ \relax 
    \{4,5,9\}, & \{4,9,7\}, & \{9,5,7\}, & {\color{red}\{2,8,a\}}, & {\color{red}\{2,a,6\}}, & {\color{red}\{a,8,6\}},\\ \relax

\end{tabular}
\]

\[
\begin{tabular}{cccccc}

    \{1,2,5\}, & \{1,4,5\}, & \{1,3,4\}, & \{2,3,4\}, &          & \{1,2,6\}, \\ \relax
    \{1,3,6\}, & \{3,5,6\}, & \{2,3,5\}, &  \\ \relax
             & \{4,7,6\}, & \{7,5,6\}, & \{2,8,6\}, &          & \{8,4,6\}, \\ \relax 
    \{4,5,9\}, & \{4,9,7\}, & \{9,5,7\}, & {\color{blue}\{2,4,a\}}, & {\color{blue}\{2,a,8\}}, & {\color{blue}\{a,4,8\}},\\ \relax

\end{tabular}
\]

We notice these complexes cannot be isomorphic because \(\{4\}\) appears in 9 facets in the first complex, but no vertex appears in more than 8 facets in the second complex. 
These complexes have an isomorphism between their facet-ridge graphs, given explicitly by corresponding facets in the lists provided. 
We also illustrate this bijection of the facets in Figure~\ref{fig_projectiveplane_graphisomorphism}.

\begin{figure}
    \centering
    \begin{tikzpicture}[scale=2,every node/.append style={circle,draw=black,fill=white,inner sep=1pt,align=center}]
\footnotesize
\pgfmathsetmacro{\rotate}{-90};
\pgfmathsetmacro{\s}{40};
\pgfmathsetmacro{\r}{0.68};
\pgfmathsetmacro{\b}{40};
\node[coordinate] at (0+\rotate:0) (456){};
\node at (0+\rotate:1) (145){145};
\node at (0+\rotate+\b:2) (134){134 \\ 125};
\node at (0+\rotate-\b:2) (125){125 \\ 134};
\node[coordinate] at (120+\rotate:1) (246){};
\node at (120+\rotate+\b:2) (126){126 \\ 234};
\node at (120+\rotate-\b:2) (234){234 \\ 126};
\node at (240+\rotate:1) (356){356};
\node at (240+\rotate+\b:2) (235){235 \\ 136};
\node at (240+\rotate-\b:2) (136){136 \\ 235};
\node[coordinate] at ($\r*(456)+{1-\r}*(145)$) (457){457 \\ 1};
\node at ($\r*(456)+{1-\r}*(246)$) (467){467};
\node at ($\r*(456)+{1-\r}*(356)$) (567){567};
\node at ($\r*(246)+{1-\r}*(456)$) (468){468};
\node[coordinate] at ($\r*(246)+{1-\r}*(126)$) (268){268 \\ 1};
\node at ($\r*(246)+{1-\r}*(234)$) (248){248 \\ 268};

\node at ($\r*(457)+{1-\r}*(145)$) (459){459};
\node[coordinate] at ($\r*(457)+{1-\r}*(567)$) (t1){};
\node[coordinate] at ($\r*(457)+{1-\r}*(467)$) (t2){};
\node at ($1.5*(t1)-0.5*(t2)$) (579){579};
\node at ($1.5*(t2)-0.5*(t1)$) (479){479};

\node at ($\r*(268)+{1-\r}*(126)$) (26a){26a \\ 24a};
\node[coordinate] at ($\r*(268)+{1-\r}*(248)$) (t3){};
\node[coordinate] at ($\r*(268)+{1-\r}*(468)$) (t4){};
\node at ($1.5*(t3)-0.5*(t4)$) (28a){28a};
\node at ($1.5*(t4)-0.5*(t3)$) (68a){68a \\ 48a};
\draw[bluishgreen] (567) -- (356);
\draw (356) -- (136);
\draw[bluishgreen] (567) -- (467) -- (468);
\draw (468) -- (68a) -- (26a);
\draw (567) -- (579) -- (459) -- (479) -- (467);
\draw (26a) -- (28a) -- (248);
\draw[bluishgreen] (248) -- (468);
\draw (28a) -- (68a);
\draw (479) -- (579);
\draw[bluishgreen] (248) -- (234);
\draw (459) -- (145) -- (125);
\draw (26a) -- (126) -- (136);
\draw (145) -- (134) -- (234);
\draw[bluishgreen] (356) -- (235);
\draw (235) -- (125);
\draw[bluishgreen] (234) to[out=60+\rotate-\s, in=90+\rotate] (0+\rotate:2.5) to[in=300+\rotate+\s,out=270+\rotate] (235);
\draw (136) to[out=180+\rotate-\s, in=210+\rotate] (120+\rotate:2.5) to[in=60+\rotate+\s,out=30+\rotate] (134);
\draw (125) to[out=300+\rotate-\s, in=330+\rotate] (240+\rotate:2.5) to[in=180+\rotate+\s,out=150+\rotate] (126);

\draw[<->,red,thick,dashed] (235) -- (136);
\draw[<->,red,thick,dashed] (234) -- (126);
\draw[<->,red,thick,dashed] (125) -- (134);
    \end{tikzpicture}
    \caption{An isomorphism between the facet-ridge graphs of the two triangulations in Figure~\ref{fig_projectiveplane_two_nonismorphic_triangulations}}
    \label{fig_projectiveplane_graphisomorphism}
\end{figure}
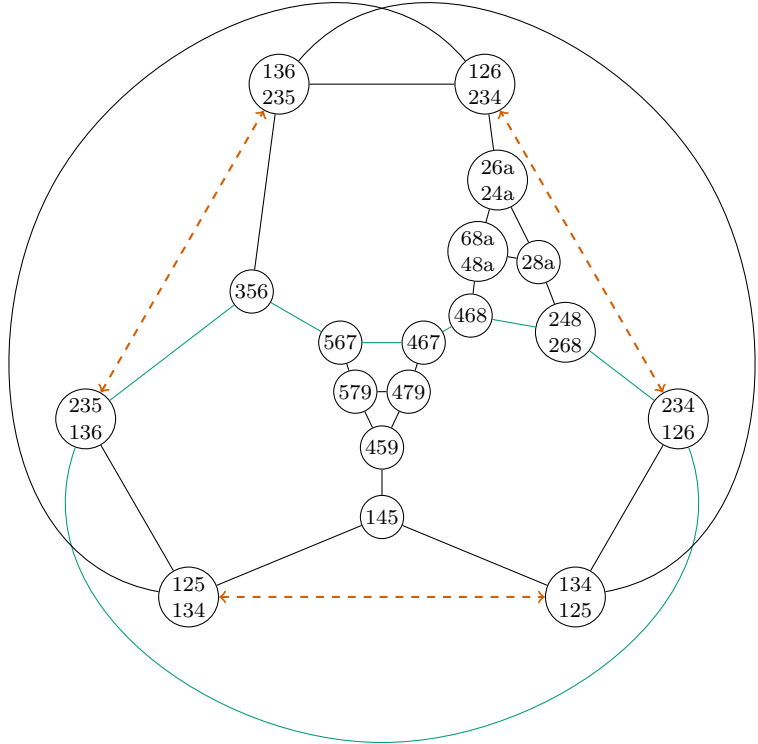

Blind and Mani's proof fails in this example, the map $g$ between these two complexes with the same facet-ridge graph is not a bijection. In particular, \(g(\{4\}) = \emptyset\).
Lemma \ref{lem_bm_8} and Theorem \ref{thm_BlindMani} do not hold. The conditions of Lemma \ref{lem_bm_4} hold, but the conclusion fails, \(\tilde{H}\) is trivial and \(g(\{4\})\) is not a \((d-2)\)-face. The reason this fails is because there are two \(1\)-chains which intersect in a single point, e.g. \(\{1,6\},\{6,5\},\{1,5\}\) and \(\{3,6\},\{4,6\},\{3,4\}\), which is impossible on a sphere.
The proofs of Lemma \ref{lem_bm_4} rely on the intersection of a \(1\)-chain and a \(d-1\)-chain not being a single point.

\subsection{Torus}
Our example for the torus follows the same idea as the projective plane. 

We begin the construction with the well-known minimal triangulation of the torus. Figure \ref{fig:min_torus} (left) represents this triangulation in such a way as to emphasize the large degree of symmetry, which is the origin of the large automorphism group of this simplicial complex and its facet-ridge graph.

\begin{figure}
    \centering
\begin{tikzpicture}[scale=0.75,every node/.append style={circle,draw=black,fill=white,inner sep=1pt,align=center}]
\footnotesize
\pgfmathsetmacro{\r}{sqrt(3)};
\node[coordinate] at (2,4*\r/3) (h1){};
\node[coordinate] at (3,-\r/3) (h2){};
\node[coordinate] at (1, -5*\r/3) (h3){};
\node[coordinate] at (-2,-4*\r/3) (h4){};
\node[coordinate] at (-3,\r/3) (h5){};
\node[coordinate] at (-1,5*\r/3) (h6){};
\draw[draw=white, fill=black!10]  (h1) -- (h2) -- (h3) -- (h4) -- (h5) -- (h6) -- cycle;

\node at (0,0) (4c){4};
\node at (2,0) (5c){5};
\node at (4,0) (6c){6};
\node at (-2,0) (3c){3};
\node at (-4,0) (2c){2};
\node at (-3,\r) (7b){7};
\node at (-1,\r) (1b){1};
\node at (1,\r) (2b){2};
\node at (3,\r) (3b){3};
\node at (-2,2*\r) (5a){5};
\node at (0,2*\r) (6a){6};
\node at (2,2*\r) (7a){7};
\node at (-3,-\r) (5d){5};
\node at (-1,-\r) (6d){6};
\node at (1,-\r) (7d){7};
\node at (3,-\r) (1d){1};
\node at (-2,-2*\r) (1e){1};
\node at (0,-2*\r) (2e){2};
\node at (2,-2*\r) (3e){3};

\path [name path=h1--h2] (h1) -- (h2);
\path [name path=2b--3b] (2b) -- (3b);
\path [name path=5c--3b] (5c) -- (3b);
\path [name path=5c--6c] (5c) -- (6c);
\path [name intersections={of=h1--h2 and 2b--3b,by=23a}];
\path [name intersections={of=h1--h2 and 5c--3b,by=35a}];
\path [name intersections={of=h1--h2 and 5c--6c,by=56a}];

\path [name path=h3--h2] (h3) -- (h2);
\path [name path=5c--1d] (5c) -- (1d);
\path [name path=7d--1d] (7d) -- (1d);
\path [name path=7d--3e] (7d) -- (3e);
\path [name intersections={of=h3--h2 and 5c--1d,by=51a}];
\path [name intersections={of=h3--h2 and 7d--1d,by=71a}];
\path [name intersections={of=h3--h2 and 7d--3e,by=73a}];

\path [name path=h3--h4] (h3) -- (h4);
\path [name path=6d--1e] (6d) -- (1e);
\path [name path=6d--2e] (6d) -- (2e);
\path [name path=7d--2e] (7d) -- (2e);
\path [name intersections={of=h3--h4 and 6d--1e,by=61a}];
\path [name intersections={of=h3--h4 and 6d--2e,by=62a}];
\path [name intersections={of=h3--h4 and 7d--2e,by=72a}];

\path [name path=h5--h4] (h5) -- (h4);
\path [name path=5d--3c] (5d) -- (3c);
\path [name path=5d--6d] (5d) -- (6d);
\path [name path=2c--3c] (2c) -- (3c);
\path [name intersections={of=h5--h4 and 5d--3c,by=53b}];
\path [name intersections={of=h5--h4 and 5d--6d,by=56b}];
\path [name intersections={of=h5--h4 and 2c--3c,by=23b}];

\path [name path=h5--h6] (h5) -- (h6);
\path [name path=5a--1b] (5a) -- (1b);
\path [name path=7b--1b] (7b) -- (1b);
\path [name path=7b--3c] (7b) -- (3c);
\path [name intersections={of=h5--h6 and 5a--1b,by=51b}];
\path [name intersections={of=h5--h6 and 7b--1b,by=71b}];
\path [name intersections={of=h5--h6 and 7b--3c,by=73b}];

\path [name path=h1--h6] (h1) -- (h6);
\path [name path=6a--1b] (6a) -- (1b);
\path [name path=6a--2b] (6a) -- (2b);
\path [name path=7a--2b] (7a) -- (2b);
\path [name intersections={of=h1--h6 and 6a--1b,by=61b}];
\path [name intersections={of=h1--h6 and 6a--2b,by=62b}];
\path [name intersections={of=h1--h6 and 7a--2b,by=72b}];

\draw (1b) -- (2b) -- (4c) -- (1b) -- (3c) -- (4c) -- (6d) -- (7d) -- (4c) -- (5c) -- (7d);
\draw (3c) -- (6d);
\draw (2b) -- (5c);
\draw[dotted] (5a) -- (6a) -- (7a) -- (3b) -- (6c) -- (1d) -- (3e) -- (2e) -- (1e) -- (5d) -- (2c) -- (7b) -- (5a);
\draw (2b) -- (23a);
\draw[dotted] (3b) -- (23a);
\draw (5c) -- (35a);
\draw[dotted] (3b) -- (35a);
\draw (5c) -- (56a);
\draw[dotted] (6c) -- (56a);

\draw (5c) -- (51a);
\draw[dotted] (1d) -- (51a);
\draw (7d) -- (71a);
\draw[dotted] (1d) -- (71a);
\draw (7d) -- (73a);
\draw[dotted] (3e) -- (73a);

\draw (6d) -- (61a);
\draw[dotted] (1e) -- (61a);
\draw (6d) -- (62a);
\draw[dotted] (2e) -- (62a);
\draw (7d) -- (72a);
\draw[dotted] (2e) -- (72a);

\draw[dotted] (2c) -- (23b);
\draw (3c) -- (23b);
\draw[dotted] (5d) -- (53b);
\draw (3c) -- (53b);
\draw[dotted] (5d) -- (56b);
\draw (6d) -- (56b);

\draw[dotted] (5a) -- (51b);
\draw (1b) -- (51b);
\draw[dotted] (7b) -- (71b);
\draw (1b) -- (71b);
\draw[dotted] (7b) -- (73b);
\draw (3c) -- (73b);

\draw[dotted] (6a) -- (61b);
\draw (1b) -- (61b);
\draw[dotted] (6a) -- (62b);
\draw (2b) -- (62b);
\draw[dotted] (7a) -- (72b);
\draw (2b) -- (72b);

\draw[bluishgreen] (1,0) arc (0:360:1);
\draw[bluishgreen] (-3.5,-\r/2) to (-1.5,-\r/2) to[out=0,in=210] (-0.5,\r/2) to[out=60, in=180] (1.5,\r/2) to (3.5,\r/2);
\end{tikzpicture}
\hspace{10pt}
\begin{tikzpicture}[scale=0.75,every node/.append style={circle,draw=black,fill=white,inner sep=1pt,align=center}]
\footnotesize
\pgfmathsetmacro{\r}{sqrt(3)};

\node[white] at (0,2*\r) (6a){};
\node[white] at (0,-2*\r) (2e){};
\node at (-1,\r/3) (134){134};
\node at (-1,-\r/3) (346){346};
\node at (1,\r/3) (245){245};
\node at (1,-\r/3) (457){457};
\node at (0,2*\r/3) (124){124};
\node at (0,-2*\r/3) (467){467};
\node at (0,-4*\r/3) (267a){267};
\node at (0,4*\r/3) (126b){126};
\node at (2,2*\r/3) (235a){235};
\node at (2,-2*\r/3) (157a){157};
\node at (2,-4*\r/3) (137a){137};
\node at (-2,2*\r/3) (137b){137};
\node at (-2,-2*\r/3) (356b){356};
\node at (-2,-4*\r/3) (156c){156};
\node at (3,-\r/3) (156a){156};
\node at (-1,5*\r/3) (156b){156};
\node at (-2,4*\r/3) (157b){157};
\node at (2,4*\r/3) (237c){237};
\node at (-3,\r/3) (237b){237};
\node at (1,-5*\r/3) (237a){237};
\node at (-3,-\r/3) (235b){235};
\node at (3,\r/3) (356a){356};
\node at (1,5*\r/3) (267b){267};
\node at (-1,-5*\r/3) (126a){126};

\draw (235a) -- (237c) -- (267b) -- (126b) -- (124) -- (245) -- (235a) -- (356a) -- (156a) -- (157a) -- (457) -- (245);
\draw (356b) -- (156c) -- (126a) -- (267a) -- (467) -- (346) -- (356b) -- (235b) -- (237b) -- (137b) -- (134) -- (346);
\draw (137b) -- (157b) -- (156b) -- (126b);
\draw (267a) -- (237a) -- (137a) -- (157a);
\draw (467) -- (457);
\draw (134) -- (124);

\draw[->,red,thick,dashed] (134) to[out=60, in=0 ] (-1,2*\r/3) to[out=180, in=120] (134);
\draw[->,red,thick,dashed] (126b) to[out=60, in=0 ] (0,5*\r/3) to[out=180, in=120] (126b);
\draw[->,red,thick,dashed] (124) to[out=240, in=180 ] (0,\r/3) to[out=0, in=300] (124);
\draw[->,red,thick,dashed] (346) to[out=240, in=180 ] (-1,-2*\r/3) to[out=0, in=300] (346);
\draw[->,red,thick,dashed] (137b) to[out=240, in=180 ] (-2,\r/3) to[out=0, in=300] (137b);
\draw[->,red,thick,dashed] (245) to[out=60, in=0 ] (1,2*\r/3) to[out=180, in=120] (245);
\draw[<->,red,thick,dashed] (356b) -- (467);
\draw[<->,red,thick,dashed] (237b) -- (157b);
\draw[<->,red,thick,dashed] (237a) -- (157a);
\draw[<->,red,thick,dashed] (156b) -- (267b);
\draw[<->,red,thick,dashed] (235a) -- (457);
\draw[<->,red,thick,dashed] (156c) -- (267a);
\draw[->,red,thick,dashed] (126a) to[out=60, in=0 ] (-1,-4*\r/3) to[out=180, in=120] (126a);
\draw[->,red,thick,dashed] (137a) to[out=240, in=180 ] (2,-5*\r/3) to[out=0, in=300] (137a);

\end{tikzpicture}
    \caption{Minimal Triangulation of the Torus (left), and its facet-ridge graph together with a graph automorphism which does not correspond to a simplicial automorphism (right).}
    \label{fig:min_torus}
\end{figure}
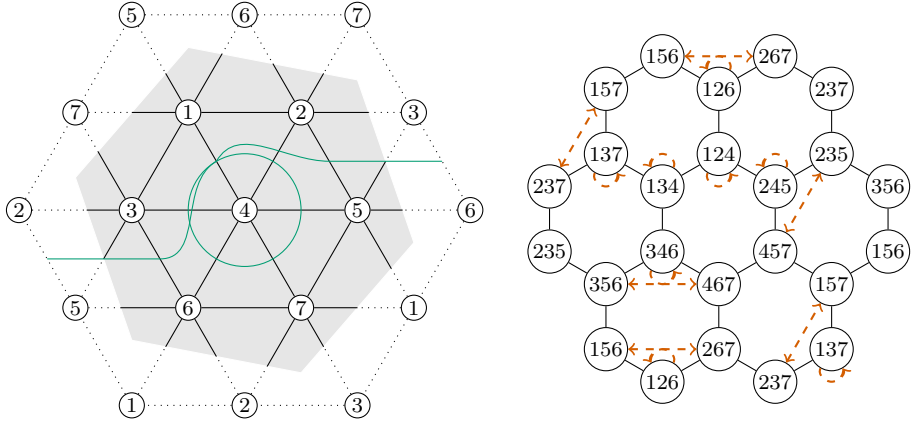

The facet list of this complex is given below

\[
\begin{tabular}{cccccc}
    \{1,2,6\}, & \{1,2,4\}, & \{1,3,4\}, & \{1,3,7\}, & \{1,5,7\}, & \{1,5,6\}, \\ \relax
    \{3,4,6\}, & \{4,6,7\}, & \{4,5,7\}, & \{2,4,5\}, & \{2,3,5\}, & \{3,5,6\}, \\ \relax
    \{2,3,7\}, & \{2,6,7\}, \\
\end{tabular}
\]

The automorphism group of this simplicial complex is of order 42, while the automorphism group of its facet-ridge graph is of order 336. 
We take advantage of this difference, and find a facet-ridge graph automorphism which does not correspond to a simplicial automorphism. 
This chosen facet-ridge graph automorphism is the involution illustrated in Figure~\ref{fig:min_torus} (right), and the two {\color{bluishgreen}green} paths drawn are transformed into each other.
The reason why this graph automorphism does not correspond to a simplicial isomorphism can be explained similarly as we did it for the projective plane (using the intersection of the triangles along the two {\color{bluishgreen}green} paths). 

We will modify this highly symmetric triangulation to preserve this graph automorphism and so that this automorphism changes the simplicial structure.
This modification will create two different simplicial complexes. First, we stellarly subdivide \(\{1,3,4\}\) with \(\{8\}\) and subdivide \(\{3,4,6\}\) with \(\{9\}\). We then stellarly subdivide \(\{1,4,8\}\) with~\(\{a\}\). Finally, to get the two different triangulations, we subdivide either \(\{4,6,9\}\) with \(\{b\}\) or \(\{3,6,9\}\) with \(\{b\}\). The resulting complex is shown in Figure~\ref{fig_torus_two_nonismorphic_triangulations}, with the two subdivisions shown in different colors. The facet lists of these two complexes appear below.

\begin{figure}
\begin{center}
\begin{tikzpicture}[scale=1.2,every node/.append style={circle,draw=black,fill=white,inner sep=1pt,align=center}]
\footnotesize
\pgfmathsetmacro{\r}{sqrt(3)};
\node[coordinate] at (2,4*\r/3) (h1){};
\node[coordinate] at (3,-\r/3) (h2){};
\node[coordinate] at (1, -5*\r/3) (h3){};
\node[coordinate] at (-2,-4*\r/3) (h4){};
\node[coordinate] at (-3,\r/3) (h5){};
\node[coordinate] at (-1,5*\r/3) (h6){};
\draw[draw=white, fill=black!10]  (h1) -- (h2) -- (h3) -- (h4) -- (h5) -- (h6) -- cycle;

\node at (0,0) (4c){4};
\node at (2,0) (5c){5};
\node at (4,0) (6c){6};
\node at (-2,0) (3c){3};
\node at (-4,0) (2c){2};
\node at (-3,\r) (7b){7};
\node at (-1,\r) (1b){1};
\node at (1,\r) (2b){2};
\node at (3,\r) (3b){3};
\node at (-2,2*\r) (5a){5};
\node at (0,2*\r) (6a){6};
\node at (2,2*\r) (7a){7};
\node at (-3,-\r) (5d){5};
\node at (-1,-\r) (6d){6};
\node at (1,-\r) (7d){7};
\node at (3,-\r) (1d){1};
\node at (-2,-2*\r) (1e){1};
\node at (0,-2*\r) (2e){2};
\node at (2,-2*\r) (3e){3};

\path [name path=h1--h2] (h1) -- (h2);
\path [name path=2b--3b] (2b) -- (3b);
\path [name path=5c--3b] (5c) -- (3b);
\path [name path=5c--6c] (5c) -- (6c);
\path [name intersections={of=h1--h2 and 2b--3b,by=23a}];
\path [name intersections={of=h1--h2 and 5c--3b,by=35a}];
\path [name intersections={of=h1--h2 and 5c--6c,by=56a}];

\path [name path=h3--h2] (h3) -- (h2);
\path [name path=5c--1d] (5c) -- (1d);
\path [name path=7d--1d] (7d) -- (1d);
\path [name path=7d--3e] (7d) -- (3e);
\path [name intersections={of=h3--h2 and 5c--1d,by=51a}];
\path [name intersections={of=h3--h2 and 7d--1d,by=71a}];
\path [name intersections={of=h3--h2 and 7d--3e,by=73a}];

\path [name path=h3--h4] (h3) -- (h4);
\path [name path=6d--1e] (6d) -- (1e);
\path [name path=6d--2e] (6d) -- (2e);
\path [name path=7d--2e] (7d) -- (2e);
\path [name intersections={of=h3--h4 and 6d--1e,by=61a}];
\path [name intersections={of=h3--h4 and 6d--2e,by=62a}];
\path [name intersections={of=h3--h4 and 7d--2e,by=72a}];

\path [name path=h5--h4] (h5) -- (h4);
\path [name path=5d--3c] (5d) -- (3c);
\path [name path=5d--6d] (5d) -- (6d);
\path [name path=2c--3c] (2c) -- (3c);
\path [name intersections={of=h5--h4 and 5d--3c,by=53b}];
\path [name intersections={of=h5--h4 and 5d--6d,by=56b}];
\path [name intersections={of=h5--h4 and 2c--3c,by=23b}];

\path [name path=h5--h6] (h5) -- (h6);
\path [name path=5a--1b] (5a) -- (1b);
\path [name path=7b--1b] (7b) -- (1b);
\path [name path=7b--3c] (7b) -- (3c);
\path [name intersections={of=h5--h6 and 5a--1b,by=51b}];
\path [name intersections={of=h5--h6 and 7b--1b,by=71b}];
\path [name intersections={of=h5--h6 and 7b--3c,by=73b}];

\path [name path=h1--h6] (h1) -- (h6);
\path [name path=6a--1b] (6a) -- (1b);
\path [name path=6a--2b] (6a) -- (2b);
\path [name path=7a--2b] (7a) -- (2b);
\path [name intersections={of=h1--h6 and 6a--1b,by=61b}];
\path [name intersections={of=h1--h6 and 6a--2b,by=62b}];
\path [name intersections={of=h1--h6 and 7a--2b,by=72b}];

\draw (1b) -- (2b) -- (4c) -- (1b) -- (3c) -- (4c) -- (6d) -- (7d) -- (4c) -- (5c) -- (7d);
\draw (3c) -- (6d);
\draw (2b) -- (5c);
\draw[dotted] (5a) -- (6a) -- (7a) -- (3b) -- (6c) -- (1d) -- (3e) -- (2e) -- (1e) -- (5d) -- (2c) -- (7b) -- (5a);
\draw (2b) -- (23a);
\draw[dotted] (3b) -- (23a);
\draw (5c) -- (35a);
\draw[dotted] (3b) -- (35a);
\draw (5c) -- (56a);
\draw[dotted] (6c) -- (56a);

\draw (5c) -- (51a);
\draw[dotted] (1d) -- (51a);
\draw (7d) -- (71a);
\draw[dotted] (1d) -- (71a);
\draw (7d) -- (73a);
\draw[dotted] (3e) -- (73a);

\draw (6d) -- (61a);
\draw[dotted] (1e) -- (61a);
\draw (6d) -- (62a);
\draw[dotted] (2e) -- (62a);
\draw (7d) -- (72a);
\draw[dotted] (2e) -- (72a);

\draw[dotted] (2c) -- (23b);
\draw (3c) -- (23b);
\draw[dotted] (5d) -- (53b);
\draw (3c) -- (53b);
\draw[dotted] (5d) -- (56b);
\draw (6d) -- (56b);

\draw[dotted] (5a) -- (51b);
\draw (1b) -- (51b);
\draw[dotted] (7b) -- (71b);
\draw (1b) -- (71b);
\draw[dotted] (7b) -- (73b);
\draw (3c) -- (73b);

\draw[dotted] (6a) -- (61b);
\draw (1b) -- (61b);
\draw[dotted] (6a) -- (62b);
\draw (2b) -- (62b);
\draw[dotted] (7a) -- (72b);
\draw (2b) -- (72b);

\node at (-1,\r/3) (8){8};
\node at (-1,-\r/3) (9){9};
\node at ($1/3*(1b)+1/3*(4c)+1/3*(8)$) (a){a};
\node[draw=blue] at ($1/3*(3c)+1/3*(6d)+1/3*(9)$) (b1){b};
\node[draw=red] at ($1/3*(4c)+1/3*(6d)+1/3*(9)$) (b2){b};
\draw (3c) -- (8) -- (4c) -- (9) -- (3c);
\draw (a) -- (8) -- (1b) -- (a) -- (4c);
\draw[red] (4c) -- (b2) -- (9);
\draw[red] (6d) -- (b2);
\draw[blue] (3c) -- (b1) -- (6d);
\draw[blue] (9) -- (b1);
\draw (9) -- (6d);

\end{tikzpicture}

\end{center}
    \caption{Two non-isomorphic triangulations of the torus with isomorphic facet-ridge graphs.}
    \label{fig_torus_two_nonismorphic_triangulations}
\end{figure}
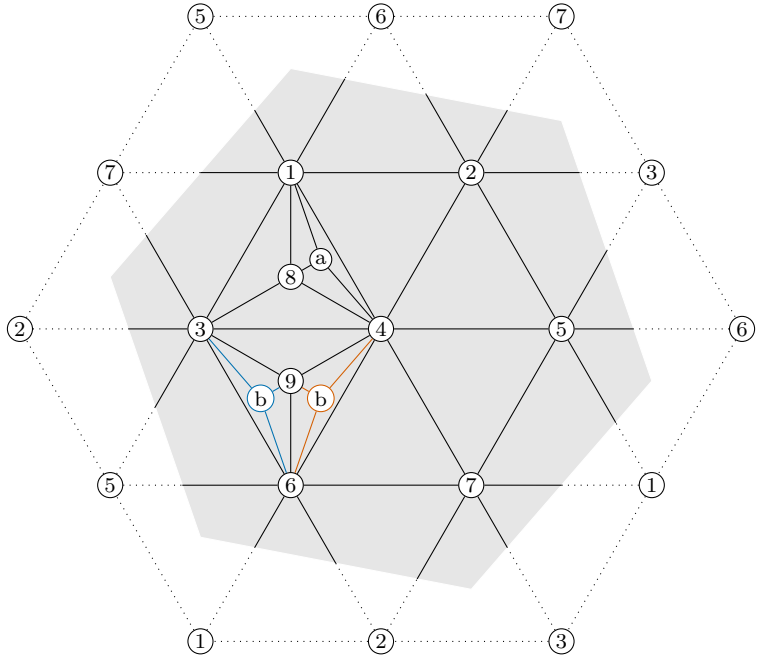

\[
\begin{tabular}{cccccc}
    \{1,2,6\}, & \{1,2,4\}, &          & \{1,3,7\}, & \{1,5,7\}, & \{1,5,6\}, \\ \relax
            & \{4,6,7\}, & \{4,5,7\}, & \{2,4,5\}, & \{2,3,5\}, & \{3,5,6\}, \\ \relax
    \{2,3,7\}, & \{2,6,7\}, \\ \relax
    \{1,3,8\}, & \{3,4,8\}, & \{1,8,a\}, & \{1,4,a\}, & \{4,8,a\}, \\ \relax
    \{3,4,9\}, & \{3,6,9\}, & {\color{red}\{4,9,b\}}, & {\color{red}\{6,9,b\}}, & {\color{red}\{4,6,b\}}, \\ \relax
\end{tabular}
\]
Above is the facet list for the complex using the {\color{red}red} copy of \(\{b\}\), and below is the facet list for the complex using the {\color{blue}blue} copy of \(\{b\}\).
\[
\begin{tabular}{cccccc}
    \{1,2,6\}, & \{1,2,4\}, &          & \{1,3,7\}, & \{2,3,7\}, & \{2,6,7\}, \\ \relax
            & \{3,5,6\}, & \{2,3,5\}, & \{2,4,5\}, & \{4,5,7\}, & \{4,6,7\}, \\ \relax
    \{1,5,7\}, & \{1,5,6\}, \\ \relax
    \{1,3,8\}, & \{3,4,8\}, & \{1,8,a\}, & \{1,4,a\}, & \{4,8,a\}, \\ \relax
    \{3,4,9\}, & \{4,6,9\}, & {\color{blue}\{3,9,b\}}, & {\color{blue}\{6,9,b\}}, & {\color{blue}\{3,6,b\}}, \\ \relax
\end{tabular}
\]

It can be seen that these complexes are not combinatorially isomorphic by looking at the number of facets containing a vertex. 
In the first complex, \(\{4\}\) is in 10 facets. 
In the second complex, every vertex is in at most 9 facets.

These complexes have isomorphic facet-ridge graphs. 
An explicit isomorphism is given by corresponding entries in the facet lists above. 
We further illustrate this isomorphism in Figure \ref{fig:torus}. 
The {\color{bluishgreen}green} line in this figure the same as the {\color{bluishgreen}green} line in Figure \ref{fig:min_torus}, where after the automorphism illustrated acts, it becomes the set of all facets containing \(\{4\}\).

\begin{figure}
    \centering

\begin{tikzpicture}[scale=1.75,every node/.append style={circle,draw=black,fill=white,inner sep=1pt,align=center}]
\footnotesize
\pgfmathsetmacro{\r}{sqrt(3)};

\node at (-1.5,\r/2) (138){138};
\node at (-1/2,\r/2) (148){14a};
\node at (-3/4,5*\r/16) (48a){48a};
\node at (-1,\r/8) (348){348};
\node at (-1,\r/2) (18a){18a};

\node at (-1.5,-\r/2) (369){36b};
\node at (-0.5,-\r/2) (469){469};
\node at (-1,-\r/8) (349){349};
\node at (-5/4,-5*\r/16) (39a){39b};
\node at (-1,-\r/2) (69a){69b};

\node at (1,\r/3) (245){245};
\node at (1,-\r/3) (457){457};
\node at (0,2*\r/3) (124){124};
\node at (0,-2*\r/3) (467){467};
\node at (0,-4*\r/3) (267a){267};
\node at (0,4*\r/3) (126b){126};
\node at (2,2*\r/3) (235a){235};
\node at (2,-2*\r/3) (157a){157};
\node at (2,-4*\r/3) (137a){137};
\node at (-2,2*\r/3) (137b){137};
\node at (-2,-2*\r/3) (356b){356};
\node at (-2,-4*\r/3) (156c){156};
\node at (3,-\r/3) (156a){156};
\node at (-1,5*\r/3) (156b){156};
\node at (-2,4*\r/3) (157b){157};
\node at (2,4*\r/3) (237c){237};
\node at (-3,\r/3) (237b){237};
\node at (1,-5*\r/3) (237a){237};
\node at (-3,-\r/3) (235b){235};
\node at (3,\r/3) (356a){356};
\node at (1,5*\r/3) (267b){267};
\node at (-1,-5*\r/3) (126a){126};

\draw (235a) -- (237c) -- (267b) -- (126b) -- (124);
\draw[bluishgreen] (124) -- (245) -- (235a) -- (356a);
\draw (356a) -- (156a) -- (157a) -- (457) -- (245);
\draw (356b) -- (156c) -- (126a) -- (267a) -- (467) -- (469) -- (349);
\draw[bluishgreen] (369) -- (356b) -- (235b);
\draw (235b) -- (237b) -- (137b) -- (138) -- (348);
\draw (137b) -- (157b) -- (156b) -- (126b);
\draw (267a) -- (237a) -- (137a) -- (157a);
\draw (467) -- (457);
\draw[bluishgreen] (369) -- (39a) -- (349) -- (348) -- (48a) -- (148) -- (124);
\draw (369) -- (69a) -- (39a);
\draw (69a) -- (469);
\draw (138) -- (18a) -- (148);
\draw (18a) -- (48a);

\draw[<->,red,thick,dashed] (356b) -- (467);
\draw[<->,red,thick,dashed] (237b) -- (157b);
\draw[<->,red,thick,dashed] (237a) -- (157a);
\draw[<->,red,thick,dashed] (156b) -- (267b);
\draw[<->,red,thick,dashed] (235a) -- (457);
\draw[<->,red,thick,dashed] (156c) -- (267a);

\end{tikzpicture}

    \caption{An isomorphism between the facet-ridge graphs of the two triangulations in Figure~\ref{fig_torus_two_nonismorphic_triangulations}.}
    \label{fig:torus}
\end{figure}
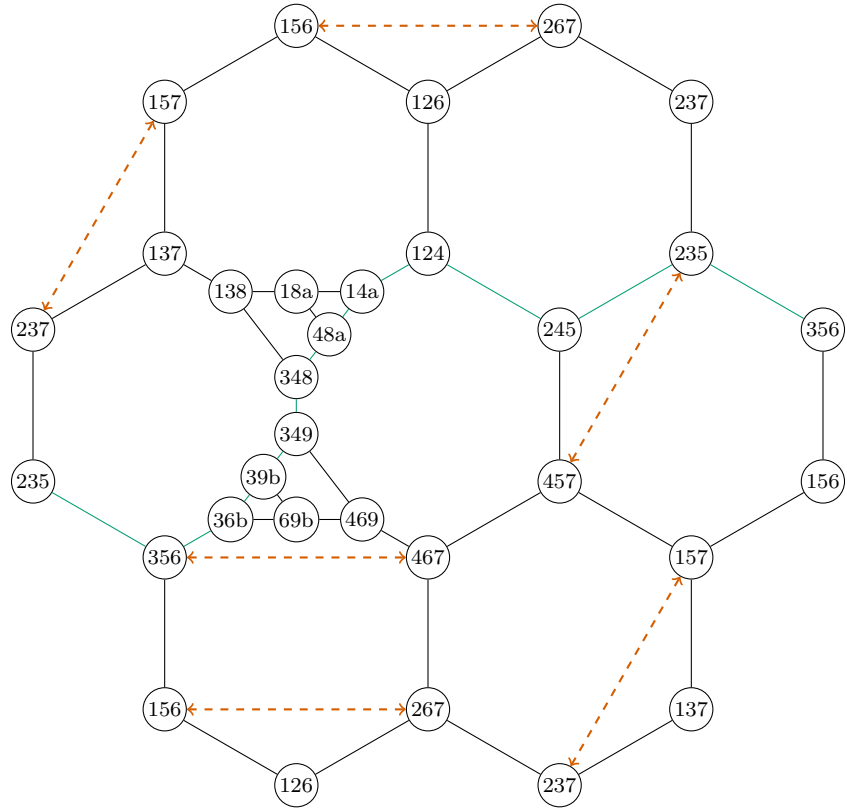

Exactly as in the projective plane case, Blind and Mani's proof fails in this example, the map $g$ between these two complexes with the same facet-ridge graph is not a bijection.  In particular, \(g(\{4\}) = \emptyset\)
Lemma \ref{lem_bm_4} even fails for exactly the same reason, there are two \(1\)-chains which intersect in a single point, e.g. \(\{2,5\},\{5,7\},\{7,2\}\) and \(\{3,4\},\{4,5\},\{3,5\}\), which is impossible on a sphere.
Lemma \ref{lem_bm_4} does not hold if the intersection of a \(1\)-chain and a \(d-1\)-chain is a single point.

\pagebreak

\section{Further Directions}\label{sec:furtherDirections}
 
We re-iterate Knutson and Miller's question in~\cite{knutson_subword_2004} and the conjecture in \cite{ceballos_subword_2014}, and ask,
\begin{question}
 Are spherical subword complexes polytopal?
\end{question}

In Section \ref{sec:nonreconstructible}, we very carefully described two different 2-manifolds that are not reconstructable from their facet-ridge graph. 
Using these as a starting point, most tame 2-manifolds can be shown to not be reconstructable, specifically those with a torus or projective plane in their connected sum decomposition.
However the 2-sphere is reconstructable~\cite[Chapter~13]{grunbaum_convexpolytopes}.
This leads us to the thought that the 2-sphere is the only 2-manifold reconstructable from its facet-ridge graph.
The natural question extending this thought and Kalai's conjecture is
\begin{question}
 Is the \(d\)-sphere the only \(d\)-manifold reconstructable from its facet-ridge graph?
\end{question}
In light of the falsehood of the triangulation conjecture, this question is of course specialized to triangulable manifolds.
We are in fact unaware of any pair of \(k\)-manifolds that share their facet-ridge graph for \(k>2\). 
This is likely due to a lack of searching, rather than the crazy alternative, that \(d\)-manifolds are reconstructable from their facet-ridge graphs when \(d>2\).

A very slight narrowing of Kalai's conjecture could focus more on the special role of polytopes.
\begin{question}
 Is there a non-polytopal simplicial sphere that shares its facet-ridge graph with a simplicial polytope?
\end{question}
It may be that Kalai's conjecture is false in general, but polytopes have distinguished graphs among all spheres.

Reconstruction from the facet-ridge graph, and reconstruction in general, must be done within a particular class. 
Our main result is reconstruction within the class of subword complexes. This leads to the following natural question. 

\begin{question}
 Is there a simplicial sphere, which is not isomorphic to any subword complex, but shares its facet-ridge graph with a spherical subword complex? 
\end{question}

Kalai's conjecture is within the class of simplicial spheres.
The previous two questions are within triangulations of a particular manifold (a sphere).
We can also take a larger class, and ask
\begin{question}
 Do there exist two simplicial complexes with the same facet-ridge graph, but whose geometric realizations are different manifolds without boundary?
\end{question}

Continuing this line of thought, it could be that the graphs of spheres are distinguished among manifolds.
\begin{question}
 Is there a non-spherical simplical manifold which shares its facet-ridge graph with a sphere?
\end{question}

 We could continue asking questions we don't know the answer to in this vein for a long time, but it may be best to leave it here.

\bibliographystyle{plain}
\bibliography{biblio}

\begin{thebibliography}{10}

\bibitem{adiprasito_2017}
Karim Adiprasito.
\newblock personal communication, 2017.

\bibitem{armstrong_sorting_2011}
Drew Armstrong and Patricia Hersh.
\newblock Sorting orders, subword complexes, {B}ruhat order and total
  positivity.
\newblock {\em Adv. in Appl. Math.}, 46(1-4):46--53, 2011.

\bibitem{barnette_sphere_1973}
David Barnette.
\newblock The triangulations of the {$3$}-sphere with up to {$8$} vertices.
\newblock {\em J. Combinatorial Theory Ser. A}, 14:37--52, 1973.

\bibitem{bergeron_hopf_2017}
Nantel Bergeron and Cesar Ceballos.
\newblock A {H}opf algebra of subword complexes.
\newblock {\em Adv. Math.}, 305:1163--1201, 2017.

\bibitem{ceballos_fan_2015}
Nantel Bergeron, Cesar Ceballos, and Jean-Philippe Labb\'{e}.
\newblock Fan realizations of type {$A$} subword complexes and
  multi-associahedra of rank 3.
\newblock {\em Discrete Comput. Geom.}, 54(1):195--231, 2015.

\bibitem{billera_decomposition_1979}
Louis~J. Billera and J.~Scott Provan.
\newblock A decomposition property for simplicial complexes and its relation to
  diameters and shellings.
\newblock In {\em Second {I}nternational {C}onference on {C}ombinatorial
  {M}athematics ({N}ew {Y}ork, 1978)}, volume 319 of {\em Ann. New York Acad.
  Sci.}, pages 82--85. New York Acad. Sci., New York, 1979.

\bibitem{CoxeterGroups}
Anders Bj\"{o}rner and Francesco Brenti.
\newblock {\em Combinatorics of {C}oxeter groups}, volume 231 of {\em Graduate
  Texts in Mathematics}.
\newblock Springer, New York, 2005.

\bibitem{blind_puzzles_1987}
Roswitha Blind and Peter Mani-Levitska.
\newblock Puzzles and polytope isomorphisms.
\newblock {\em Aequationes Math.}, 34(2-3):287--297, 1987.

\bibitem{brown_1960}
Morton Brown.
\newblock A proof of the generalized {S}choenflies theorem.
\newblock {\em Bull. Amer. Math. Soc.}, 66:74–76, 1960.

\bibitem{ceballos_subword_2014}
Cesar Ceballos, Jean-Philippe Labb{\'e}, and Christian Stump.
\newblock Subword complexes, cluster complexes, and generalized
  multi-associahedra.
\newblock {\em J. Algebraic Combin.}, 39(1):17--51, 2014.

\bibitem{ceballos_vtamari_subword_2020}
Cesar Ceballos, Arnau Padrol, and Camilo Sarmiento.
\newblock The {$\nu$}-{T}amari lattice via {$\nu$}-trees, {$\nu$}-bracket
  vectors, and subword complexes.
\newblock {\em Electron. J. Combin.}, 27(1):Paper No. 1.14, 31, 2020.

\bibitem{ceballos_denominator_2015}
Cesar Ceballos and Vincent Pilaud.
\newblock Denominator vectors and compatibility degrees in cluster algebras of
  finite type.
\newblock {\em Trans. Amer. Math. Soc.}, 367(2):1421--1439, 2015.

\bibitem{chapoton_generalizedssociahedra_2002}
Fr\'{e}d\'{e}ric Chapoton, Sergey Fomin, and Andrei Zelevinsky.
\newblock Polytopal realizations of generalized associahedra.
\newblock {\em Canad. Math. Bull.}, 45(4):537--566, 2002.
\newblock Dedicated to Robert V. Moody.

\bibitem{escobar_bott_2014}
Laura Escobar.
\newblock Bott-{S}amelson varieties, subword complexes and brick polytopes.
\newblock In {\em 26th {I}nternational {C}onference on {F}ormal {P}ower
  {S}eries and {A}lgebraic {C}ombinatorics ({FPSAC} 2014)}, Discrete Math.
  Theor. Comput. Sci. Proc., AT, pages 863--874. Assoc. Discrete Math. Theor.
  Comput. Sci., Nancy, 2014.

\bibitem{escobar_brick_2016}
Laura Escobar.
\newblock Brick manifolds and toric varieties of brick polytopes.
\newblock {\em Electron. J. Combin.}, 23(2):Paper 2.25, 18, 2016.

\bibitem{escobar_toric_2016}
Laura Escobar and Karola M\'{e}sz\'{a}ros.
\newblock Toric matrix {S}chubert varieties and their polytopes.
\newblock {\em Proc. Amer. Math. Soc.}, 144(12):5081--5096, 2016.

\bibitem{escobar_subword_rootpolytopes_2018}
Laura Escobar and Karola M\'{e}sz\'{a}ros.
\newblock Subword complexes via triangulations of root polytopes.
\newblock {\em Algebr. Comb.}, 1(3):395--414, 2018.

\bibitem{firsching_completeenumeration_2020}
Moritz Firsching.
\newblock The complete enumeration of 4-polytopes and 3-spheres with nine
  vertices.
\newblock {\em Israel J. Math.}, 240(1):417--441, 2020.

\bibitem{fomin_ysystems_2003}
Sergey Fomin and Andrei Zelevinsky.
\newblock {$Y$}-systems and generalized associahedra.
\newblock {\em Ann. of Math. (2)}, 158(3):977--1018, 2003.

\bibitem{friedman_finding_2009}
Eric~J. Friedman.
\newblock Finding a simple polytope from its graph in polynomial time.
\newblock {\em Discrete Comput. Geom.}, 41(2):249--256, 2009.

\bibitem{goodman_pollack_asymptoticallyfewerpolytopes_1986}
Jacob~E. Goodman and Richard Pollack.
\newblock There are asymptotically far fewer polytopes than we thought.
\newblock {\em Bull. Amer. Math. Soc. (N.S.)}, 14(1):127--129, 1986.

\bibitem{goodman_pollack_upperbounds_1986}
Jacob~E. Goodman and Richard Pollack.
\newblock Upper bounds for configurations and polytopes in {${\bf R}^d$}.
\newblock {\em Discrete Comput. Geom.}, 1(3):219--227, 1986.

\bibitem{grunbaum_convexpolytopes}
Branko Gr\"{u}nbaum.
\newblock {\em Convex polytopes}, volume 221 of {\em Graduate Texts in
  Mathematics}.
\newblock Springer-Verlag, New York, second edition, 2003.
\newblock Prepared and with a preface by Volker Kaibel, Victor Klee and
  G\"{u}nter M. Ziegler.

\bibitem{grunbaum_sphere_1967}
Branko Gr\"{u}nbaum and V.~P. Sreedharan.
\newblock An enumeration of simplicial {$4$}-polytopes with {$8$} vertices.
\newblock {\em J. Combinatorial Theory}, 2:437--465, 1967.

\bibitem{hachimori_ziegler_spheresknots_2000}
Masahiro Hachimori and G\"{u}nter~M. Ziegler.
\newblock Decompositons of simplicial balls and spheres with knots consisting
  of few edges.
\newblock {\em Math. Z.}, 235(1):159--171, 2000.

\bibitem{jahn_minkowski_2021}
Dennis Jahn, Robert L\"{o}we, and Christian Stump.
\newblock Minkowski decompositions for generalized associahedra of acyclic
  type.
\newblock {\em Algebr. Comb.}, 4(5):757--775, 2021.

\bibitem{jonsson_generalized_2005}
Jakob Jonsson.
\newblock Generalized triangulations and diagonal-free subsets of stack
  polyominoes.
\newblock {\em J. Combin. Theory Ser. A}, 112(1):117--142, 2005.

\bibitem{joswig_Ksystems_2002}
Michael Joswig, Volker Kaibel, and Friederike K\"{o}rner.
\newblock On the {$k$}-systems of a simple polytope.
\newblock {\em Israel J. Math.}, 129:109--117, 2002.

\bibitem{kalai_manyspheres_1988}
Gil Kalai.
\newblock Many triangulated spheres.
\newblock {\em Discrete Comput. Geom.}, 3(1):1--14, 1988.

\bibitem{kalai_simple_1988}
Gil Kalai.
\newblock A simple way to tell a simple polytope from its graph.
\newblock {\em J. Combin. Theory Ser. A}, 49(2):381--383, 1988.

\bibitem{kalai_blog_2009}
Gil Kalai.
\newblock Telling a simple polytope from its graph, Blog available at
  \href{https://gilkalai.wordpress.com/2009/01/16/telling-a-simple-polytope-from-its-graph/}{https://gilkalai.wordpress.com/2009/01/16/telling-a-simple-polytope-from-its-graph/}.
  Jan 2009.

\bibitem{knutson_subword_2004}
Allen Knutson and Ezra Miller.
\newblock Subword complexes in {C}oxeter groups.
\newblock {\em Adv. Math.}, 184(1):161--176, May 2004.

\bibitem{knutson_grobner_2005}
Allen Knutson and Ezra Miller.
\newblock {Gr\"obner geometry of Schubert polynomials}.
\newblock {\em Ann. Math. (2)}, 161(3):1245--1318, 2005.

\bibitem{labbe_combinatorialfoundations_2020}
Jean-Philippe Labb{\'e}.
\newblock Combinatorial foundations for geometric realizations of subword
  complexes of coxeter groups.
\newblock ArXiv Preprit
  \href{https://arxiv.org/abs/2003.02753}{https://arxiv.org/abs/2003.02753},
  March 2020.

\bibitem{lutz_nonconstructiblespheres_2004}
Frank~H. Lutz.
\newblock Small examples of nonconstructible simplicial balls and spheres.
\newblock {\em SIAM J. Discrete Math.}, 18(1):103--109, 2004.

\bibitem{lutz_manifoldsfewvertices}
Frank~H. Lutz.
\newblock Combinatorial 3-manifolds with 10 vertices.
\newblock {\em Beitr\"{a}ge Algebra Geom.}, 49(1):97--106, 2008.

\bibitem{many_spheresfewvertices_1972}
P.~Mani.
\newblock Spheres with few vertices.
\newblock {\em J. Combinatorial Theory Ser. A}, 13:346--352, 1972.

\bibitem{manneville_fan_2018}
Thibault Manneville.
\newblock Fan realizations for some 2-associahedra.
\newblock {\em Exp. Math.}, 27(4):377--394, 2018.

\bibitem{mohar_vodopivec_2006}
BOJAN MOHAR and ANDREJ VODOPIVEC.
\newblock On polyhedral embeddings of cubic graphs.
\newblock {\em Combinatorics, Probability and Computing}, 15(6):877–893,
  2006.

\bibitem{morse_1960}
Marston Morse.
\newblock A reduction of the {S}choenflies extension problem.
\newblock {\em Bull. Amer. Math. Soc.}, 66:113–115, 1960.

\bibitem{pfeifle_ziegler_manythreespheres}
Julian Pfeifle and G\"{u}nter~M. Ziegler.
\newblock Many triangulated 3-spheres.
\newblock {\em Math. Ann.}, 330(4):829--837, 2004.

\bibitem{pilaudpocchiola_multitraingulations_2012}
Vincent Pilaud and Michel Pocchiola.
\newblock Multitriangulations, pseudotriangulations and primitive sorting
  networks.
\newblock {\em Discrete Comput. Geom.}, 48(1):142--191, 2012.

\bibitem{pilaudstump_brick_2015}
Vincent Pilaud and Christian Stump.
\newblock Brick polytopes of spherical subword complexes and generalized
  associahedra.
\newblock {\em Adv. Math.}, 276:1--61, 2015.

\bibitem{pilaud_barycenter_2015}
Vincent Pilaud and Christian Stump.
\newblock Vertex barycenter of generalized associahedra.
\newblock {\em Proc. Amer. Math. Soc.}, 143(6):2623--2636, 2015.

\bibitem{rote_expansive_2003}
G\"{u}nter Rote, Francisco Santos, and Ileana Streinu.
\newblock Expansive motions and the polytope of pointed pseudo-triangulations.
\newblock In {\em Discrete and computational geometry}, volume~25 of {\em
  Algorithms Combin.}, pages 699--736. Springer, Berlin, 2003.

\bibitem{serrano_maximal_2012}
Luis Serrano and Christian Stump.
\newblock Maximal fillings of moon polyominoes, simplicial complexes, and
  {S}chubert polynomials.
\newblock {\em Electron. J. Combin.}, 19(1):Paper 16, 18, 2012.

\bibitem{stump_newperspective_2011}
Christian Stump.
\newblock A new perspective on {$k$}-triangulations.
\newblock {\em J. Combin. Theory Ser. A}, 118(6):1794--1800, 2011.

\end{thebibliography}

\end{document}